\documentclass[11pt,reqno]{amsart}
\usepackage{amsmath,latexsym,amssymb,amsfonts}
\usepackage{amscd,amssymb,epsfig}
\usepackage{graphicx}
\usepackage[pagebackref, colorlinks = true, linkcolor = blue, urlcolor  = blue, citecolor = red]{hyperref}

\usepackage[margin=1in]{geometry}

\newtheorem{theorem}{Theorem}[section]
\newtheorem{lemma}[theorem]{Lemma}

\newtheorem{prop}[theorem]{Proposition}
\newtheorem{proposition}[theorem]{Proposition}
\newtheorem{corollary}[theorem]{Corollary}
\newtheorem{example}[theorem]{Example}

\newtheorem{definition}[theorem]{Definition}
\newtheorem{remark}[theorem]{Remark}

\newcommand\NCP{\mathcal{NC}}
\newcommand\ee{\varepsilon}

\newcommand\EE{\mathbb{E}}
\newcommand\FF{\mathbb{F}}
\newcommand\CC{\mathbb{C}}
\newcommand\RR{\mathbb{R}}
\newcommand\cA{\mathcal{A}}
\newcommand\cD{\mathcal{D}}
\newcommand\cB{\mathcal{B}}
\newcommand\cR{\mathcal{R}}

\newcommand\SP{\mathcal{P}}

\title[On the asymptotic distribution of block-modified random matrices]{On the asymptotic distribution of block-modified\\random matrices}

\author[O. Arizmendi]{Octavio Arizmendi}
\address{Department of Probability and Statistics, CIMAT, Guanajuato, Mexico }
\email{octavius@cimat.mx }

\author[I. Nechita]{Ion Nechita}
\address{Zentrum Mathematik, M5\\ Technische Universit\"at M\"unchen\\ Boltzmannstrasse 3, 85748 Garching\\ Germany; CNRS, Laboratoire de Physique Th\'{e}orique, IRSAMC\\ Universit\'{e} de Toulouse, UPS\\ F-31062 Toulouse, France}
\email{nechita@irsamc.ups-tlse.fr}

\author[C. Vargas]{Carlos Vargas}
\address{Technische Universit\"{a}t Graz, Department of Mathematical Structure Theory , Steyrergasse 30 III, 8010-Graz, Austria }
\email{obieta@math.tugraz.at}

\date{\today}
\begin{document}

\begin{abstract}
We study random matrices acting on tensor product spaces which have been transformed by a linear block operation. Using operator-valued free probability theory, under some mild assumptions on the linear map acting on the blocks, we compute the asymptotic eigenvalue distribution of the modified matrices in terms of the initial asymptotic distribution. Moreover, using recent results on operator-valued subordination, we present an algorithm that computes, numerically but in full generality, the limiting eigenvalue distribution of the modified matrices. Our analytical results cover many cases of interest in quantum information theory: we unify some known results and we obtain new distributions and various generalizations.

\end{abstract}

\maketitle

\tableofcontents

\section{Introduction}

We continue the work of Banica and second named author \cite{BaNe12b} on finding the asymptotic distribution (as $d\to\infty$) of $dm\times dm$ random matrices which have been block-wise modified by some fixed self-adjoint linear transformation $\varphi: M_m(\CC)\to  M_n(\CC)$. More concretely, we consider a self-adjoint $dm\times dm$ unitarily invariant random matrix $X_d$ and a linear map $\varphi: M_m(\CC)\to M_n(\CC)$. Our goal is to understand the asymptotic eigenvalue distribution of the block-wise modified random matrix $$X_d^{\varphi}:=[\mathrm{id}_d\otimes\varphi](X_d) \in M_d(\mathbb C) \otimes M_n(\mathbb C).$$

The motivation for this line of research comes from quantum information theory. The problem of deciding whether a quantum state $\rho \in M_n(\mathbb C) \otimes M_m(\mathbb C)$ is \emph{entangled} has been shown to be equivalent to the fact that, for all positivity preserving maps $\varphi:M_m(\CC)\to  M_n(\CC)$, the modified state $\hat \rho = [\mathrm{id} \otimes \varphi](\rho)$ is positive semi-definite \cite{horodecki1996separability}. It is thus important to understand how the spectrum of generic matrices behaves under linear block-transformations. Aubrun \cite{Au12} studied the positivity of partially transposed random quantum states: this corresponds to the transposition map acting on the blocks of a Wishart element. He computed the range of parameters of the Wishart distribution for which the partial transposition is still positive. Similar computations, in different asymptotic regimes, or for different maps acting on the blocks, have been performed in \cite{BaNe12,fukuda2013partial,JLN14,JLN15}. In this paper, we unify these results by showing that operator-valued free probability theory provides the right framework to study such questions.

Our new approach relies on the theory of operator-valued free probability, developed by Voiculescu \cite{Vo85,Vo95} and later by Speicher \cite{Sp98}. We find a general numerical solution using the tools from \cite{BSTV14} which allow to approximate operator-valued free multiplicative convolutions. For certain cases, we are able to get explicit formulas for the distributions by using the operator-valued $S$-transform \cite{Dy06} over some suitable commutative algebras.

For the case $m=n$, the main observation is that we may write $$X_d^{\varphi}=\sum^m_{i,j,k,l=1}\alpha^{ij}_{kl}  (I_d  \otimes E_{ij})  X (I_d  \otimes E_{kl}),$$
where $E_{ij}\in M_m(\CC)$ are the matrix units. The joint non-commutative distribution (see Section \ref{OFP} for the main definitions) of the matrices $(I_d  \otimes E_{ij})^m_{i,j=1}$ (with respect to the state $\tau_{dm}:=\EE\circ \frac{1}{dm}\mathrm{Tr}$) does not change with $d$. Thus, if the random matrix $X$ is unitarily invariant (for example, if $X$ is a Wishart, Wigner, or randomly rotated matrix), we are allowed to use Voiculescu's asymptotic freeness results \cite{Vo91} to compute the asymptotic joint moments of the non-commutative random variables $X$, $(I_d  \otimes E_{ij})^m_{i,j=1}$ as $d\to \infty$. These asymptotic joint moments are given by replacing $(I_d  \otimes E_{jl})^m_{i,j=1}$ and $X$ by an abstract collection of operators $(e_{ij})_{i,j\leq m}$, $x$ in some non-commutative probability space $(\cA,\tau)$.

The limiting distribution of $X_d^{\varphi}$ is the same as the distribution of the element
$$x^{\varphi}:=\sum^m_{i,j,k,l=1}\alpha^{ij}_{kl}  e_{ij}xe_{kl}.$$
Thus, the problem is reduced to study the distribution of such elements for which the machinery of operator-valued free probability is available. Note that the general case $m\neq n$ is solved similarly with the aid of rectangular spaces, using the theory developed by Benaych-Georges \cite{benaych2007infinitely,BG09,BG09b}.

The paper is organized as follows. We start with a short overview of operator-valued free probability, the framework in which we state all our results. Section \ref{BLT} contains the numerical solution to the problem, in full generality. We then move to the theoretical study: in Section \ref{sec:freeness} we prove a freeness result, under some assumptions on the eigenvectors of the Choi matrix $C$ of the linear map $\varphi$ acting on the blocks. The exact distribution of the modified matrix is computed in Section \ref{sec:distribution}, under more stringent assumptions on $C$. Sections \ref{sec:examples-distributions} and \ref{sec:examples-maps} contain some important examples, both of distributions and of maps $\varphi$; most of the examples are of interest in Quantum Information Theory.

\section{Operator-valued free probability}\label{OFP}

Free probability was introduced by Voiculescu \cite{Vo85} in order to tackle some problems in operator algebras. Later, the theory detached from its operator-algebraic origins, and grew into a branch of non-commutative probability, where the concept of \emph{freeness} replaces the notion of independence from classical probability.  The main object of study are (non-commutative) random variables $a_1,\dots,a_k$ in a $C^*$-algebra $\cA$ and their joint distribution with respect to a given state (i.e. a unital positive linear functional $\tau:\cA\to\CC$). The pair $(\cA,\tau)$ is called a non-commutative probability space.

By the joint distribution of an ordered tuple $a=(a_1,\dots,a_k)$ of random variables in $(\cA,\tau)$, we mean the collection of maps $\Phi(a)=\bigcup_{r\geq 0}\Phi_r^a,$ where  $$\Phi_r^{a}:=\Phi_r=\{(i_1,\dots,i_r)\mapsto \tau(a_{i_1}\dots a_{i_r}): i_1,\dots,i_r\leq k\}$$ are the $r$-th order mixed moments of $(a_1,\dots,a_k)$.

The state $\tau$ should be thought as playing the role of the expectation functional in classical probability. Motivated by operator-algebraic constructions, Voiculescu defined freeness in \cite{Vo85}, as a new, non-commutative notion of independence. In \cite{Vo91}, free random variables where constructed as limits of certain large random matrices, establishing a connection between free probability theory and random matrix theory.

The idea of operator-valued free probability (\cite{Vo95}) is to generalize free probability, by replacing the scalar valued linear form $\tau:\cA\to\CC$ by a conditional expectation $\EE:\cA\to\cB$ onto a larger sub-algebra $\CC\subseteq\cB\subseteq\cA$.
This leads to a broader definition of freeness, which occurs in more general situations of random matrix theory, as observed already by Shlyakthenko in \cite{Sh96}.

Many aspects of the theory of (scalar-valued) free probability can be lifted to the operator-valued level. The combinatorics of operator-valued free probability (see \cite{Sp98}) remains the same provided that the nesting structure of non-crossing partitions is respected while operating.
This will make cumulants particularly handy when finding good candidates for the smallest sub-algebra $\cB \subset \cA$ over which two given random variables are free.

The question of finding explicit formulas for operator-valued distributions is closely related to the possibility of finding realizations of the distribution in terms of operators which are free over a commutative algebra. We use the criteria in \cite{NSS02} to give sufficient conditions (in terms of the Choi matrix of the map $\varphi$) for such a realization to exist.

For the cases where $\cB$ is non-commutative it is extremely complicated to obtain exact distributions. On the other hand, numerical algorithms which rely on subordination, such as the ones in \cite{BeBe07} for the computation of the additive and multiplicative free convolutions admit very effective generalizations to the operator-valued level (see \cite{BSTV14,BMS13}). We will use these to obtain a general numerical solution to our problem.

\subsection{Basic definitions}

We gather here some basic definitions that will be needed in what follows, and we establish some notation.

\begin{definition}
(1) Let $\mathcal{A}$ be a unital $*$-algebra and let $\CC\subseteq\mathcal{B}\subseteq\cA$ be a $*$-subalgebra.
A $\mathcal{B}${-probability space} is a pair $\left(\mathcal{A},\EE\right)$, where  $\EE:\mathcal{A}\to\mathcal{B}$ is a {conditional expectation}, that is, a linear map satisfying:
\begin{eqnarray*}
\EE\left(bab'\right) &=& b\EE(a)b', \qquad \forall b,b'\in\mathcal{B},a\in\mathcal{A}\\
\EE\left(1\right) &=& 1.
\end{eqnarray*}

(2) Let $(\cA,\EE)$ be a $\cB$-probability space and let $\bar{a}:=a-\EE(a)1_{\cA}$ for any $a\in\cA$.
The $*$-subalgebras $\cB\subseteq A_1,\dots ,A_k\subseteq \cA$ are $\cB${-free} (or {free over} $\cB$, or {free with amalgamation over} $\cB$) ({with respect to} $\EE$) iff
\begin{equation}\label{opfreeness}
\EE(\bar{a_1}\bar{a_2}\cdots \bar{a_r})=0,
\end{equation}
for all $r\geq 1$ and all tuples $a_1,\dots,a_r\in \cA$ such that $a_i\in A_{j(i)}$ with $j(1)\neq j(2)\neq \dots \neq j(r)$.

(3) Subsets $S_1,\dots ,S_k\subset \cA$ are $\cB${-free} if so are the $*$-subalgebras $\langle S_1,\cB\rangle,\dots,\langle S_k,\cB\rangle$.
\end{definition}

Similar to independence, freeness allows to compute mixed moments free random variables in terms of their individual moments.

The most basic examples of non-commutative probability spaces are the space of deterministic matrices $(M_n(\CC),\frac{1}{n}\mathrm{Tr})$ and the space of classical random variables $(\cA,\EE)$. These serve as building blocks for more general (operator-valued) non-commutative probability spaces.

\begin{example}[Rectangular probability spaces]
Let $\cA$ be a unital $C^*$-algebra and let $\tau:\cA\to \CC$ be a state. Let $p_1,\dots,p_k$ be pairwise orthogonal projections with $1=p_1+\dots+p_k$ and $\tau(p_i) >0$ for all $i$. There exist a unique conditional expectation $\EE:\cA\to \langle p_1,\dots,p_k\rangle$ compatible with $\tau$ in the sense that $\tau\circ \EE=\tau$. The conditional expectation is explicitly given by $$\EE(a)=\sum_{i\leq k}\tau(p_i)^{-1}\tau(p_iap_i).$$
\end{example}

\begin{definition} \label{convdist}
Let $(\cA_N,\tau_N)$, $N\geq 1$, and $(\cA,\tau)$ be $\CC$-probability spaces and let $(a_1^{(N)},\dots,a_k^{(N)})\in \cA_N^k$, $(a_1,\dots,a_k)\in \cA^k$ be such that $$\lim_{N\to\infty}\tau_N((a_{i_1}^{(N)})\cdots (a_{i_r}^{(N)}))=\tau(a_{i_1}\cdots a_{i_r}),$$ for all $r\geq 1$, $1\leq i_1,\dots, i_m\leq k$ Then we say that $(a_1^{(N)},\dots,a_k^{(N)})$ converges in distribution to $(a_1,\dots,a_k)$ and we write $(a_1^{(N)},\dots,a_k^{(N)})\to(a_1,\dots,a_k)$.
\end{definition}

\begin{remark}\label{freemunits}
(1) The joint distribution (with respect to $\tau_{md}$) of the blown-up matrix units  $(I_d\otimes E_{ij} )_{i,j\leq m}$ is independent of $d$. We may consider an abstract ``limiting'' non-commutative probability space $(\cA,\tau)$ such that $M_m(\CC)\subseteq \cA$ and $\tau|_{M_m(\CC)}=\frac{1}{m}\mathrm{Tr}$. If $(e_{ij})_{i,j\leq m}$ denote the matrix units in $M_m(\CC)\subseteq \cA$ we have that
$$(E_{11} \otimes I_d,E_{12}\otimes I_d,\dots ,E_{mm}\otimes I_d)\to (e_{11},e_{12},\dots ,e_{mm}),$$
with joint distribution determined by the usual rules: $$e_{ij}e_{kl}=\delta_{jk}e_{il}, \quad \tau(e_{ij})=n^{-1}\delta_{ij}, \quad \sum_{i=1}^me_{ii}=1, \quad e_{ij}^*=e_{ji}.$$

(2) If $X_d$ is an $md\times md$ Wigner, Wishart or randomly rotated self-adjoint matrix with $\lim_{d\to \infty}\tau_{dn}(X_d^k)=\tau(x^k)$ for some fixed operator $x$ in a non-commutative probability space $(\cA, \tau)$, then by \cite{Vo91}, we have that
 $$(X, E_{11} \otimes I_d,E_{12}\otimes I_d,\dots ,E_{mm}\otimes I_d)\to (x, e_{11},e_{12},\dots ,e_{mm}),$$
where $x$ and $(e_{ij})_{i,j\leq m}$ are free. This determines completely the joint distribution of these variables.

(3) Let us now replace $m$ by $m+n$ in the first part of this remark so that our limiting non-commutative probability space $(\cA,\tau)$ contains now a copy of $(M_{(m+n)}(\CC),\frac{1}{m+n}\mathrm{Tr})$. If $X_d$ is still an $md\times md$ random matrix, as in part (2) (completed with zeros to an $d(m+n)\times d(m+n)$ matrix), then we have again, that
$$(X_d, E_{11} \otimes I_d,E_{12}\otimes I_d,\dots ,E_{(m+n)(m+n)}\otimes I_d)\to (x,e_{11},e_{12},\dots ,e_{(m+n)(m+n)}).$$
It is quite easy to see that $x$ and $((e_{ij})_{i,j\leq m+n})$ are no longer $\CC$-free. Indeed, the joint distribution is aware of the orthogonality relations between $X_d$ and the different matrix units, and we have, for example, that $$0=\tau(x^2e_{(m+1)(m+1)})\neq\tau(x^2)\tau(e_{(m+1)(m+1)}),$$ and hence freeness is broken unless $x=0$.

As noticed in \cite{BG09b}, this problem gets fixed by considering distributions with values on the commutative algebra generated by the projections
$$P_{d,m}:=\sum_{i\leq m} I_d\otimes E_{ii}, \quad P_{d,n}:=\sum_{m< i \leq m+n} I_d\otimes  E_{ii},$$
which in turn converge in distribution to the orthogonal projections $$p_m := \sum_{i=1}^m e_{ii} \quad \text{ and } \quad p_n := \sum_{i=m+1}^{m+n} e_{ii},$$ with non zero traces $\tau(p_m) = m/(m+n)$, $\tau(p_n) = n/(m+n)$. Then $x$ and $((e_{ij})_{i,j\leq m+n})$ are $\langle p_m,p_n\rangle$-free.
\end{remark}

\begin{example}[Matrix-valued probability spaces]
Let $\cA$ be a unital $C^*$-algebra and let $\tau:\cA\to \CC$ be a state. Consider the algebra $M_n(\cA)\cong M_n(\CC)\otimes\cA$ of $n\times n$ matrices with entries in $\cA$.
The maps $$\EE_3:(a_{ij})_{ij}\mapsto(\tau(a_{ij}))_{ij}\in M_n(\CC),$$ $$\EE_2:(a_{ij})_{ij}\mapsto(\delta_{ij}\tau(a_{ij}))_{ij}\in D_n(\CC),$$ and $$\EE_1:(a_{ij})_{ij}\mapsto\sum_{i=1}^n\frac{1}{n}\tau(a_{ii})I_n\in \CC\cdot I_n$$ are respectively, conditional expectations onto the algebras $M_n(\CC)\supset D_n(\CC)\supset \CC\cdot I_n$ of constant matrices, diagonal matrices and multiples of the identity.
\end{example}

If $A_1,\dots ,A_k$ are free in $(\cA,\tau)$, then the algebras $M_n(A_1),\dots,M_n(A_k)$ of matrices with entries in $A_1,\dots,A_k$ respectively are in general not free over $\CC$ (with respect to $\EE_1$).
They are, however $M_n(\CC)$-free (with respect to $\EE_3$).
Below is a slightly more general assertion of this simple but fundamental result.

\begin{prop}\label{matrixfreeness}
Let $(\cA,\EE)$ be a $\cB$-probability space, and consider the $M_n(\cB)$-valued probability space $(M_n(\CC)\otimes\cA,id\otimes\EE)$.
If $A_1,\dots,A_k\subseteq\cA$ are $\cB$-free, then $(M_n(\CC)\otimes A_1),\dots,(M_n(\CC)\otimes A_k)\subseteq(M_n(\CC)\otimes\cA)$ are $(M_n(\cB))$-free.
\end{prop}
\begin{proof}
Let $a^{(1)},\dots,a^{(m)}\in M_n(\CC)\otimes\cA$ be such that $a^{(i)}\in M_n(\CC)\otimes A_{j(i)}$ with $j(1)\neq j(2)\neq \dots \neq j(m)$.
Observe that $$\overline{a^{(i)}}=a^{(i)}-(\mathrm{id}\otimes \EE)(a^{(i)})=((a^{(i)}_{rs})-\EE(a^{(i)}_{rs}))_{rs\leq n}=(\overline{a^{(i)}_{rs}})_{rs\leq n}.$$
Hence
\begin{equation}
(id\otimes\EE)((\overline{a^{(1)}})\cdots(\overline{a^{(m)}}))=\sum_{i_0,\dots,i_m=1}^n(\EE((\overline{a^{(1)}_{i_0i_1}})(\overline{a^{(2)}_{i_1i_2}})\cdots(\overline{a^{(m)}_{i_{m-1}i_m}})))_{i_0i_m}=0.
\end{equation}
\end{proof}

\subsection{Combinatorics and cumulants}

A \emph{partition} of a set is a decomposition into disjoint subsets,
called \emph{blocks}.
The set of partitions of the set $[n]:=\{1,\ldots,n\}$ is denoted by
$\SP(n)$. Any partition defines an equivalence relation
on $[n]$ and vice versa: given  $\pi\in \SP(n)$,
$i \sim_\pi j$ holds if and only if there is a block $V \in \pi$
such that $i,j \in V$.

A partition $\pi\in\SP(n)$ is \emph{non-crossing} if there is no
quadruple of elements $1\leq i<j<k<l\leq n$ such that $i\sim_\pi k$, $j\sim_\pi l$ and
$i\not\sim_\pi j$. The non-crossing partitions of order $n$ form a
sub-poset of $\SP(n)$ which we denote by $\NCP(n)$.

For $n\in\mathbb{N}$, a $\mathbb{C}$-multi-linear map $f:\mathcal{A}^{n}\to\mathcal{B}$ is called $\mathcal{B}$\textit{-balanced} if it satisfies the $\mathcal{B}$-bilinearity conditions, that for all $b,b'\in\mathcal{B}$, $a_{1},\dots,a_{n}\in\mathcal{A}$, and for all $r=1,\dots,n-1$
\begin{eqnarray*}
f\left(ba_{1},\dots,a_{n}b'\right) &=& bf\left(a_{1},\dots,a_{n}\right)b'\\
f\left(a_{1},\dots,a_{r}b,a_{r+1},\dots,a_{n}\right) &=& f\left(a_{1},\dots,a_{r},ba_{r+1}\dots,a_{n}\right)
\end{eqnarray*}

A collection of $\mathcal{B}$-balanced maps $\left(f_{\pi}\right)_{\pi\in \NCP}$ is said to be {multiplicative} with respect to the lattice of non-crossing partitions if, for every $\pi\in \NCP$, $f_{\pi}$ is computed using the block structure of $\pi$ in the following way:

\begin{itemize}
 \item If $\pi=\hat{1}_{n}\in \NCP\left(n\right)$, we just write $f_{n}:=f_{\pi}$.

 \item If $\hat{1}_{n}\neq\pi=\left\{ V_{1},\dots,V_{k}\right\} \in \NCP\left(n\right),$ then by a known characterization of $\NCP$, there exists a block $V_{r}=\left\{ s+1,\dots,s+l\right\} $ containing consecutive elements. For any such a block we must have
\begin{equation*}
f_{\pi}\left(a_{1},\dots,a_{n}\right)=f_{\pi\backslash V_{r}}\left(a_{1},\dots,a_{s}f_{l}\left(a_{s+1},\dots,a_{s+l}\right),a_{s+l+1},\dots,a_{n}\right),
\end{equation*}
where $\pi\backslash V_{r}\in \NCP\left(n-l\right)$ is the partition obtained from removing the block $V_{r}$.

\end{itemize}

We observe that a multiplicative family $\left(f_{\pi}\right)_{\pi\in \NCP}$ is entirely determined by $\left(f_{n}\right)_{n\in\mathbb{N}}$. On the other hand, every collection $\left(f_{n}\right)_{n\in\mathbb{N}}$ of $\mathcal{B}$-balanced maps can be extended uniquely to a multiplicative family $\left(f_{\pi}\right)_{\pi\in \NCP}$.

The {operator-valued free cumulants} $\left(R^{\cB}_{\pi}\right)_{\pi\in \NCP}$ are indirectly and inductively defined as the unique multiplicative family of $\mathcal{B}$-balanced maps satisfying the (operator-valued) moment-cumulant formulas
\begin{equation*}
\EE\left(a_{1}\dots a_{n}\right)=\sum_{\pi\in \NCP\left(n\right)}R^{\cB}_{\pi}\left(a_{1},\dots,a_{n}\right)
\end{equation*}

By the {cumulants of a tuple} $a_{1},\dots,a_{k}\in\mathcal{A}$, we mean the collection of all cumulant maps
\begin{equation*}
\begin{array}{cccc}
R_{i_{1},\dots,i_{n}}^{\mathcal{B};a_{1},\dots,a_{k}}: & \mathcal{B}^{n-1} & \to & \mathcal{B},\\
 & \left(b_{1},\dots,b_{n-1}\right) & \mapsto & R^{\cB}_{n}\left(a_{i_{1}},b_{1}a_{i_2},\dots,b_{i_{n-1}}a_{i_n}\right)\end{array}
\end{equation*}
for $n\in\mathbb{N}$, $1\leq i_{1},\dots,i_{n}\leq k$.

A cumulant map $R_{i_{1},\dots,i_{n}}^{\mathcal{B};a_{1},\dots,a_{k}}$ is {mixed} if there exists $r<n$ such that $i_{r}\ne i_{r+1}$. The main feature of the operator-valued cumulants is that they characterize freeness with amalgamation:

\begin{theorem}[\cite{Sp98}]
The random variables $a_{1},\dots,a_{n}$ are $\mathcal{B}$-free iff all their mixed cumulants vanish.
\end{theorem}

We recall a useful result from \cite{NSS02} which gives conditions for (operator-valued) cumulants to be restrictions of cumulants with respect to a larger algebra.

\begin{prop}

\label{thm:Restrictions}

Let $1\in\mathcal{D}\subset\mathcal{B}\subset\mathcal{A}$ be algebras such that $(\mathcal{A},\FF)$ and $(\mathcal{B},\EE)$ are respectively $\mathcal{B}$-valued and $\mathcal{D}$-valued probability spaces (and therefore $(\mathcal{A},\EE\circ\FF)$ is a $\mathcal{D}$-valued probability space) and let $a_{1},\dots,a_{k}\in\mathcal{A}$.

Assume that the $\mathcal{B}$-cumulants of $a_{1},\dots,a_{k}\in\mathcal{A}$ satisfy
\begin{equation*}
R_{i_{1},\dots,i_{n}}^{\mathcal{B};a_{1},\dots,a_{k}}\left(d_{1},\dots,d_{n-1}\right)\in\mathcal{D},
\end{equation*}
for all $n\in\mathbb{N}$, $1\leq i_{1},\dots,i_{n}\leq k$, $d_{1},\dots,d_{n-1}\in\mathcal{D}$.

Then the $\mathcal{D}$-cumulants of $a_{1},\dots,a_{k}$ are exactly the restrictions of the $\mathcal{B}$-cumulants of $a_{1},\dots,a_{k}$, namely
\begin{equation*}
R_{i_{1},\dots,i_{n}}^{\mathcal{B};a_{1},\dots,a_{k}}\left(d_{1},\dots,d_{n-1}\right)=R_{i_{1},\dots,i_{n}}^{\mathcal{D};a_{1},\dots,a_{k}}\left(d_{1},\dots,d_{n-1}\right),
\end{equation*}
for all $n\in\mathbb{N}$, $1\leq i_{1},\dots,i_{n}\leq k$, $d_{1},\dots,d_{n-1}\in\mathcal{D}$.
\end{prop}

\begin{corollary}\label{maintool}
Let $\cB\subseteq A_1,A_2\subseteq \cA$ be $\cB$-free and let $\cD\subseteq M_N(\CC)\otimes\cB$. Assume that, individually, the $M_N(\CC)\otimes\cB$-valued moments (or, equivalently, the $M_N(\CC)\otimes\cB$-cumulants) of both $x\in M_N(\CC)\otimes A_1$ and $y\in M_N(\CC)\otimes A_2$, when restricted to arguments in $\cD$, remain in $\cD$. Then $x,y$ are $\cD$-free.
\end{corollary}

\begin{proof}
By Proposition \ref{matrixfreeness} $x,y$ are $M_N(\CC)\otimes\cB$ free. The $M_N(\CC)\otimes\cB$-mixed moments on $x,y$ restricted to arguments $d_1,\dots,d_{n-1}\in \cD$ can be expressed through the $M_N(\CC)\otimes\cB$ moment-cumulant formula $$\EE\left(a_{1}d_1\dots d_{n-1}a_{n}\right)=\sum_{\pi\in \NCP\left(n\right)}R^{M_N(\CC)\otimes\cB}_{\pi}\left(a_{1},\dots,d_{n-1}a_{n}\right),$$  as a sum of (nested) products of individual $M_N(\CC)\otimes\cB$ cumulants on $x$ and $y$. For each partition $\pi$, we pick an interval block. The cumulant corresponding to this block remains in $\cD$ (in particular it is zero if the cumulant is mixed). By proceeding inductively removing interval blocks, we obtain that each sumand is in $\cD$, and hence, so is any restricted mixed moment $\EE\left(a_{1}d_1\dots d_{n-1}a_{n}\right)$.

By Moebius inversion (or induction), any restricted mixed cumulant is also in $\cD$ and we are allowed to use Theorem \ref{thm:Restrictions}. Hence the mixed $\cD$ cumulants of $x,y$ coincide with the restrictions of the $M_N(\CC)\otimes\cB$ cumulants, which vanish due to $M_N(\CC)\otimes\cB$-freeness. Thus the $\cD$-mixed cumulants vanish as well, proving the assertion.
\end{proof}

In view of Remark \ref{freemunits} (3), if $A_1:=\langle(e_{ij})_{i,j\leq m+n}\rangle$ and $A_2:=\langle x\rangle$ are $\langle p_m,p_n\rangle$ free, then (by Proposition \ref{matrixfreeness}) $M_r(\CC)\otimes A_1$ and $M_r(\CC)\otimes A_2$ are $M_r(\CC)\otimes \langle p_m,p_n\rangle$-free. Our goal is to replace $M_r(\CC)\otimes \langle p_m,p_n\rangle$-freeness by $\cB$-freeness for some commutative subalgebra $\cB\subset M_r(\CC)\otimes \langle p_m,p_n\rangle$, using Corollary \ref{maintool}.

\subsection{Transforms}

Like in the scalar-valued case, there are analytical tools to compute operator-valued free convolutions, which are based on the $\mathcal{B}$-valued Cauchy transform $$G_x^{\mathcal{B}}(b)=\EE((b-x)^{-1}),$$ which maps the operatorial upper half-plane $$\mathbb{H}^+(\mathcal{B}):=\{b\in\mathcal{B}:\exists \ee> 0 \text{ such that }-i(b-b^*)\geq \ee\cdot 1\}$$ into the lower half-plane $\mathbb{H}^-(\mathcal{B})=-\mathbb{H}^+(\mathcal{B})$.
In the usual settings coming from random matrix models (as we have seen above), our probability space $\cA$ may have several operator-valued structures $\EE_i:\cA\to\cB_i$ simultaneously, with $\CC=\cB_1\subset\cB_2\subset\dots\subset\cB_k$, and $\EE_i\circ \EE_{i+1}=\EE_i$.
We are usually interested ultimately in the scalar-valued distribution, which can be obtained (via Stieltjes-inversion) from the Cauchy transform.
The later can be obtained from any ''upper'' $\mathcal{B}_i$-valued Cauchy transform, as we have that, for all $b\in\cB_i$ $$\EE_i(G_x^{\mathcal{B}_{i+1}}(b))=\EE_i\circ \EE_{i+1}((b-x)^{-1})=\EE_i((b-x)^{-1})=G^{\cB_i}_x(b).$$

In terms of moments, the operator-valued Cauchy transform is given by $$G_x^{\mathcal{B}}(b)=\EE((b-x)^{-1})=\sum_{n\geq 0}\EE(b^{-1}(xb^{-1})^n).$$

The {operator-valued $\mathcal{R}$-transform} is defined by
\begin{equation*}
\mathcal{R}_{x}^{\mathcal{B}}\left(b\right)=\sum_{n\geq1}R_{n}^{\mathcal{B}}\left(x,bx,\dots,bx\right).
\end{equation*}
The vanishing of mixed cumulants for free variables implies the additivity of the cumulants, and thus also the additivity of the $\cR$-transforms (\cite{Vo95}): If $a_1$ and $a_2$ are $\cB$-free then we have for $b\in \cB$ that
$\cR^{\mathcal{B}}_{a_1+a_2}(b)=\cR^{\mathcal{B}}_{a_1}(b)+\cR_{a_2}(b)$.

These transforms satisfy the functional equation
\begin{equation}
G_{a}^{\mathcal{B}}\left(b\right)=\left(\mathcal{R}_{a}^{\mathcal{B}}\left(G_{a}^{\mathcal{B}}\left(b\right)\right)-b\right)^{-1} \label{GR}.
\end{equation}

The $\mathcal S$-transform was defined by Voiculescu \cite{Vo87} to compute multiplicative free convolutions. The most general operator valued generalization of the $\mathcal S$ transform was defined in \cite{Dy06}. 

If $M_a(b)=\sum_{k\geq 1}\EE((ab)^k)$ is the moment generating series of $a$, it is possible to define a compositional inverse $M_a^{\langle -1\rangle}(b)$ in a suitable domain using the analytic inverse function theorems on Banach algebras. The $\mathcal S$-transform is then defined as $\mathcal S_a(b)=(1+b)b^{-1}M_a^{\langle -1\rangle}(b)$.

The $\mathcal S$-transform factorizes the operator-valued multiplicative convolution in the following sense: If $x,y$ are $\cB$-free, then $$\mathcal S_{xy}(b)=\mathcal S_y(b)\mathcal S_x(\mathcal S_y(b)^{-1}b\mathcal S_y(b)).$$

In this work we will be mainly interested in the case where the algebra $\cB$ is commutative, the right hand can be then simplified to $\mathcal S_x(b)\mathcal S_y(b)$. 

For our general numerical solution, where $\cB$ is not commutative, the factorization property of the $\mathcal S$-transform is implicitly being used in the subordination formulas (see \cite{BSTV14}).

In the commutative case, we will use the following equation relating the $\mathcal S$ and the $\cR$-transforms: $b\mathcal R(b) + \mathcal S(b \mathcal R(b)) = b$ (see, for example, Lecture 18 in \cite{NiSp06}).

A drawback of the operator-valued setting is that, unless we ask $\mathcal{B}$ to be commutative, one can hardly compute explicit distributions:
although we have $\mathcal{B}$-valued generalizations of the $\mathcal R$ and $\mathcal S$-transforms, the task of explicitly inverting these operator-valued analytic maps is nearly impossible for any non-trivial situation (even for finite dimensional, relatively simple sub-algebras, like $\cB=M_2(\mathbb{C})$).

A very powerful (numerical) method to obtain $\mathcal{B}$-valued free convolutions is based on the analytic subordination phenomena observed by Biane (\cite{Bi98}, see also \cite{Vo02}). In particular, the approach of \cite{BeBe07} to obtain the subordination functions by iterating analytic maps can be very efficiently performed in the $\mathcal{B}$-valued context.

In \cite{BSTV14} the fixed point algorithm from \cite{BeBe07} for computing free multiplicative convolutions was generalized to the operator-valued level. Let us first recall the eta transform (or Boolean cumulant transform)

$$\eta_x(b)=1-b(\EE((b^{-1}-x)^{-1}))^{-1}.$$

\begin{theorem}[\cite{BSTV14}]\label{Main-op}
Let $x>0,y=y^*\in\mathcal A$ be two random
variables with invertible expectations, free over $\cB$.
There exists a
Fr\'echet holomorphic map $
\omega_2\colon\{b\in \cB\colon\Im(bx)>0\}\to\mathbb H^+(\cB),$
such that
\begin{enumerate}
\item $\eta_y(\omega_2(b))=\eta_{xy}(b)$, $\Im(bx)>0$;
\item $\omega_2(b)$ and $b^{-1}\omega_2(b)$ are analytic around zero;
\item For any $b\in \cB$ so that $\Im(bx)>0$, the map $g_b\colon
\mathbb H^+(\cB)\to\mathbb H^+(\cB)$, $g_b(w)=bh_x(h_y(w)b)$, where
\begin{equation}
h_x(b)=b^{-1}-\mathbb E\left[(b^{-1}-x)^{-1}\right]^{-1};
\end{equation}
is well-defined, analytic and
for any fixed $w\in\mathbb H^+(\cB)$,
$$
\omega_2(b)=\lim_{n\to\infty}g_b^{\circ n}(w),
$$
in the weak operator topology.
\end{enumerate}
Moreover, if one defines $\omega_1(b):=h_y(\omega_2(b))b$, then
$$
\eta_{xy}(b)=\omega_2(b)\eta_x(\omega_1(b))\omega_2(b)^{-1},
\quad\Im (bx)>0.
$$
\end{theorem}

The invertibility condition on $\EE(y)$ can be dropped if we restrict to finite dimensional algebras.

\begin{prop}[\cite{BSTV14}]\label{main-finite}
Let $\cB$ be finite-dimensional and $x>0$, $y=y^*$ be free
over $\cB$. There exists a function $g\colon\{b\in \cB\colon\Im(bx)>0\}\times\mathbb H^+(\cB)
\to\mathbb H^+(\cB)$ so that
\begin{enumerate}
\item $\omega_2(b):=\lim_{n\to\infty}g_b^{\circ n}(w)$ exists, does not depend on $w\in
\mathbb H^+(\cB)$, and is analytic on $\{b\in \cB\colon\Im(bx)>0\}$;
\item
$$
\eta_y(\omega_2(b))=\eta_{xy}(b),\quad b\in\{b\in \cB\colon\Im(bx)>0\}.
$$
\end{enumerate}
\end{prop}

From a numerical perspective, switching between $h_x$ to $G_x$ is a simple operation.
By Proposition \ref{main-finite} one only needs the individual $\mathcal{B}$-valued Cauchy transforms of $x,y$ (or good approximations of these) to obtain the $\mathcal{B}$-valued Cauchy transform of $xy$, and hence, its probability distribution.

In a scalar-valued non-commutative probability space $(\cA,\tau)$, we have the integral representation of the Cauchy transform:
$$G_x(z)=\tau((z-x)^{-1})=\int_{\RR}(z-t)^{-1}d\mu_x(t).$$
Analogously, for linear, self-adjoint elements $c\otimes x$ in a $M_n(\CC)$-valued probability space $(M_n(\CC)\otimes(\cA),\mathrm{id}_m\otimes \tau)$, we have:
\begin{equation}\label{Rsum}
G_{c\otimes x}(b)=(\mathrm{id}_m\otimes \tau)((b-c\otimes x)^{-1})=\int_{\RR}(b-c\otimes t)^{-1}d\mu_x(t).                                                                                                                                                                                                          \end{equation}
The previous integrals can be approximated, for example, by using matrix-valued Riemann sums.

The case of deterministic matrices is simpler. If we assume that $M_n(\CC)\subset\cA$ and consider $x=x^*\in M_m(\CC)\otimes M_n(\CC)$. Then $G^{\cB}_{x}(b)=G^{\cB}_{x}(b\otimes I_n)$ is just the partial trace of the resolvent
\begin{equation}\label{parttrace}
G^{\cB}_{x}(b)=(\mathrm{id}_m\otimes \frac{1}{n}\mathrm{Tr})((b\otimes I_n-x)^{-1}).                                                                                                                                                                                                           \end{equation}

\subsection{Unitarily invariant random matrices: definition and examples}
\label{sec:random-matrices}

In this work, our main focus is the effect of a block-linear transformation on the (asymptotical) spectrum of a random matrix. The models of random matrices we consider are as follows.

\begin{definition} A self-adjoint random matrix $X \in M_n(\mathbb C)$ is called \emph{unitarily invariant} if, for all unitary operators $U \in \mathcal U_n$, the random matrices $X$ and $UXU^*$ have the same distribution.
\end{definition}

Let us consider two special cases of the definition above: the GUE ensemble and the Wishart ensemble (see \cite{AGZ10} for the general theory of random matrices).

Both GUE and Wishart ensembles are built out of i.i.d.~complex Gaussian random variables. For any integer size parameter $n$, consider the self-adjoint random matrix $G_n$ having upper diagonal entries given by
$$G_n(i,j) = \begin{cases}
Z_{ii}/\sqrt{n},&\qquad \text{ if } i=j\\
(Z_{ij} + \tilde Z_{ij})/\sqrt{2n},&\qquad \text{ if } i < j,
\end{cases}$$
where $Z_{ij}$ and $\tilde Z_{ij}$ are i.i.d.~standard complex Gaussian random variables. The classical Wigner theorem \cite[Theorem 2.2.1]{AGZ10} states that the non-commutative random variables $G_n$ converge in distribution to the standard semicircle distribution

$$\mu_{sc;0,1} = \frac{\sqrt{4-x^2}}{2\pi} \mathbf{1}_{[-2,2]}(x) dx.$$

More generally, the semicircle distribution with average $a \in \mathbb R$ and variance $\sigma^2 >0$ is defined by
$$\mu_{sc;a,\sigma} = \frac{\sqrt{4\sigma^2-(x-a)^2}}{2\pi \sigma^2} \mathbf{1}_{[a-2\sigma ,a+ 2\sigma]}(x) dx.$$

The Wishart ensemble is defined starting from a rectangular random matrix $X_n \in M_{n \times k}(\mathbb C)$, having i.i.d.~standard complex Gaussian entries. A Wishart matrix of parameters $(n,k)$ is then $W_n = X_n X_n^*$. Consider now, for a sequence of parameters $k_n$ such taht $k_n / n \to c>0$, a sequence of random Wishart matrices $W_n$ of parameters $(n,k_n)$. The limiting eigenvalue distribution of the sequence $W_n$, as $n \to \infty$, is given by the Mar\u{c}enko-Pastur (or free Poisson) distribution of parameter $c$ (see \cite[Exercise 2.1.18]{AGZ10})
$$\pi_c=\max (1-c,0)\delta_0+\frac{\sqrt{4c-(x-1-c)^2}}{2\pi x}1_{[1+c-2\sqrt{c},1+c+2\sqrt{c}]}dx.$$
The family of free Poisson measures $\pi_c$ are characterized, in terms of free probability, by the fact that their free cumulants are all equal to $c$. This family admits the following extension. For a compactly supported probability measure $\mu$, define the compound free Poisson distribution of parameter $\mu$ as the unique probability measure $\pi_\mu$ having free cumulants given by
$$\forall n \geq 1, \qquad R_n(\pi_\mu) = \int_{\mathbb R} x^n d\mu(x).$$

Finally, let us introduce the model of unitarily invariant random matrices we are interested in, and of which both the GUE and the Wishart ensembles are examples of. For a compactly supported probability measure on the real line $\mu$, we associate a sequence of self-adjoint random matrices $X_n \in M_n(\mathbb C)$ defined as
$$X_n = U_n D_n U_n^*,$$
where $D_n$ is a sequence of deterministic, diagonal matrices converging in distribution to $\mu$, and $U_n \in \mathcal U_n$ is a sequence of Haar-distributed random unitary matrices. In this paper, we are going to suppose that the space on which $X_n$ acts has a tensor product structure, and we are going to study the eigenvalue distribution of the \emph{modified} version of $X_n$, obtained by acting with a fixed linear map $\varphi$ on the blocks of $X_n$.

In view of Remark \ref{freemunits}, in order to understand the asymptotic distribution of a block-modified unitarily invariant random matrix ensemble, it seems important to study first the asymptotic joint distribution of its blocks $(x_{ij})$, where $x_{ij}=e_{1i}xe_{j1}$. By Remark \ref{freemunits}, this joint distribution depends only on $m$ and $\mu$.

In Lecture 20 of \cite{NiSp06} it was observed that, with respect to the compressed state $\tau_{11}(a)=m\tau(e_{11}ae_{11})$, the joint distribution of $(x_{ij})$ is completely characterized by their mixed cumulants, which satisfy the R-cyclic property:
\begin{equation} \label {rcyclic}
 R^{\tau_{11}}_k(x_{i_1j_1},x_{i_2j_2},\dots,x_{i_{k}j_{k}})=\left\{\begin{array}{ll}
m^{-(k-1)}r_k(x,x,\dots,x), & \text{if } j_1=i_2, j_2=i_3,\dots,j_{k}=i_1,\\
0,& \text{otherwise}
                                                      \end{array}\right.
\end{equation}
where $r_n(x,x,\dots,x)$ are the free cumulants of $x$ (or $\mu$) in the original space. This implies that, more generally, for any $\pi\in\NCP(k)$,
$$R^{\tau_{11}}_{\pi}(x_{i_1j_1},x_{i_2j_2},\dots,x_{i_{k}i_{k}})=
m^{-(k-|\pi|)}r_{\pi}(x,x,\dots,x),
$$
whenever the cyclic condition $j_{v_1}=i_{v_2},\dots, j_{v_{|V|}}=i_{v_1}$ holds for each block $V=\{v_1,\dots,v_{|V|}\}\in \pi$, and zero otherwise.

The distribution of a block-modified ensemble will then be that of the element $$x^{\varphi}:=\sum_{i,j\leq m}c_{ij}\otimes x_{ij}\in M_n(\CC)\otimes \cA_{11}$$ with respect to the functional $\varphi:=\mathrm{tr}_n \otimes \tau_{11}$, where $c_{ij}=\varphi(e_{ij})$ are the blocks of the Choi matrix of $\varphi:M_m(\CC)\to M_n(\CC)$ (see Section \ref{sec:freeness}).

For example, the distribution of each diagonal sumand $c_{ii}\otimes x_{ii}=c_{ii}^*\otimes x_{ii}^*$ is given by the \emph{classical} multiplicative convolution of the discrete measure $\mu_{c_{ii}}$ with the distribution $\mu_{x_{11}}$ of the free compression of $x$. Furthermore, the R-cyclic condition (\ref{rcyclic}) implies that the mixed cumulants of diagonal elements vanish. Hence $x_{11},\dots,x_{mm}\in \cA_{11}$ are identically distributed and free among themselves.

For a general situation we get that the moments of $x^{\varphi}$ are given by
\begin{eqnarray}
\varphi((x^{\varphi})^k)&=&\sum_{i_1,j_1,\dots,i_k,j_k\leq n}\phi((c_{i_1j_1}\otimes x_{i_1j_1})\cdots (c_{i_kj_k}\otimes x_{i_kj_k}))\\
&=&\sum_{i_1,j_1,\dots,i_k,j_k\leq n}\mathrm{tr}_n(c_{i_1j_1}\cdots c_{i_kj_k})\tau_{11}(x_{i_1j_1}\cdots  x_{i_kj_k})\\
&=&\sum_{i_1,j_1,\dots,i_k,j_k\leq n}\mathrm{tr}_n(c_{i_1j_1}\cdots c_{i_kj_k})\sum_{\pi\in \NCP(n)}R^{\tau_{11}}_{\pi}(x_{i_1j_1}\cdots  x_{i_kj_k}).
\end{eqnarray}

Hence the joint distribution of the Choi blocks and the blocks of unitarily invariant random matrices interact in a non-trivial way. One of the main objectives in this paper is to find conditions on the Choi matrices so that we are able to compute the distribution of $x^{\varphi}$ explicitly.

\section{General numerical solution}\label{BLT}

Now we show that the free multiplicative convolution can be used to compute the distribution of random matrices which have been deformed block-wise by a fixed self-adjoint linear transformation. For simplicity and clarity of exposition, we only solve the case $n=m$ (from our treatment of the case $m\neq n$ in the subsequent Sections, it should be clear how to adapt our algorithm to cover the general case).

We have already pointed out that the limiting distribution of $X_d^{\varphi}$ is the same as the distribution of $$x^{\varphi}:=\sum^m_{i,j,k,l=1}\alpha^{ij}_{kl}  e_{ij}xe_{kl},$$ where $(e_{ij})_{ij\leq m}$ are matrix units, free from $x$, and $\alpha^{ij}_{kl}$ are fixed constants. The self-adjointness condition on $\varphi$ is equivalent to the condition $\alpha^{ij}_{kl}=\bar \alpha^{lk}_{ji}$, for all  $i,j,k,l \leq m$.

If $\varphi$ is self-adjoint, it will be useful to express it as a difference of positive maps. For this we consider the lexicographic order of two-letter words in an $m$-letter alphabet, and write, for each $(i,j)<(l,k)$,
\begin{eqnarray*}
re^{i\theta}e_{ij}xe_{kl}+re^{-i\theta}e_{lk}xe_{ji}&=&r^{1/2}(e^{i\theta/2}e_{ij}+e^{-i\theta/2}e_{lk})xr^{1/2}(e^{-i\theta/2}e_{ji}+e^{i\theta/2}e_{kl})\\
& &-re_{ij}xe_{ji}-re_{lk}xe_{kl} \\
&=& (f^{ij}_{kl})x(f^{ij}_{kl})^*-re_{ij}xe_{ji}-re_{lk}xe_{kl}
\end{eqnarray*}
so that we get
\begin{eqnarray*}
x^{\varphi}&=&\sum_{\substack{1\leq i,j,k,l\leq m\\(i,j)<(l,k)}} (f^{ij}_{kl})x(f^{ij}_{kl})^*+\sum_{1\leq i,j \leq m}\beta_{ij}e_{ij}xe_{ji}\\
&=&\sum_{\substack{1\leq i,j,k,l\leq m\\(i,j)<(l,k)}} (f^{ij}_{kl})x(f^{ij}_{kl})^*+\sum_{1\leq i,j \leq m}f_{ij}(-1)^{\varepsilon(i,j)}xf_{ij}^*\\
&=&\sum_{\substack{1\leq i,j,k,l\leq m\\(i,j)\leq(l,k)}} (f^{ij}_{kl})\varepsilon^{ij}_{kl}x(f^{ij}_{kl})^*,
\end{eqnarray*}
where $\varepsilon^{ij}_{ji}=(-1)^{\varepsilon(i,j)}$ and $\varepsilon^{ij}_{kl}=1$ for $(i,j)\neq(l,k)$.

From the elements $f^{ij}_{kl}$, $(i,j)\leq(l,k)$ we build a vector $f=(f^{11}_{11}, f^{11}_{12}, \dots ,f^{mm}_{mm})$ of size $N:=m^2(m^2+1)/2$. We consider also the diagonal matrix $\Sigma=\mathrm{diag}(\varepsilon^{11}_{11},\varepsilon^{11}_{12},\dots, \varepsilon^{mm}_{mm})$ and we let  $\tilde{x}:=\Sigma \otimes x$.

We see that $f\tilde{x}f^*=x^{\varphi}$, so the desired distribution is the same (modulo a Dirac mass at zero of weight $1-1/N$) as the distribution of $f^*f\tilde{x}$ in the $C^*$-probability space $(M_N(\CC) \otimes \cA,\frac{1}{N}\mathrm{Tr}_N\otimes \tau)$. Moreover, since $w$ and each of the $e_{ij}$ are free, by Proposition \ref{matrixfreeness}, the matrices $f^*f$ and $\tilde{x}$ are free with amalgamation over $\cB=M_{N}(\CC)\otimes 1$ (with respect to the conditional expectation $\EE:=\mathrm{id}_N\otimes \tau$).

By Proposition \ref{main-finite}, we can obtain the $M_{N}(\CC)$-valued Cauchy transform of $f\tilde{x}f^*$ numerically, provided that we can compute the $M_{N}(\CC)$-valued Cauchy transforms (or good approximations) of $f^*f$ and $\tilde{x}$. The Cauchy transforms of the elements  $\tilde{x}$ and $f^*f$ can be respectively computed as explained in equations (\ref{Rsum}) and (\ref{parttrace}).

\subsection{Numerical simulations}

We compare the distributions obtained with our method and the empirical eigenvalue distributions of $20$ realizations of $1000\times1000$ block-modified for the cases where $x$ has a Wigner, Wishart and arcsine limiting distribution and the block transformation

$$\varphi((a_{ij})_{i,j\leq 2})=\begin{pmatrix}
11a_{11}+15a_{22}-25a_{12}-25a_{21} & 36a_{21} \\
36a_{12} & 11a_{11}-4a_{22}
\end{pmatrix}$$

The numerical results are plotted in Figures \ref{fig:Wigner},  \ref{fig:Wishart}, and  \ref{fig:arcsine}. 

\begin{figure}[ht]
\begin{center}
\epsfig{file=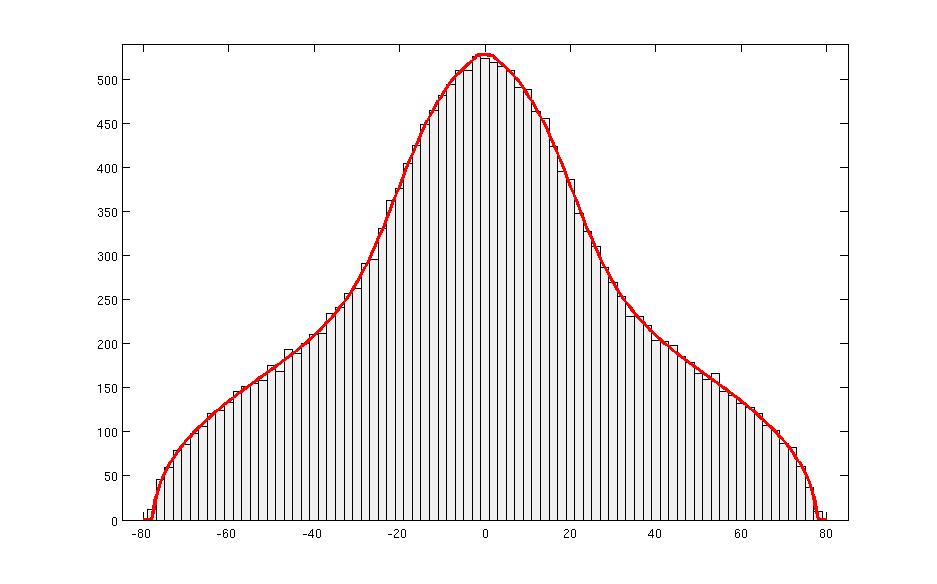, width=10cm}
\caption{Eigenvalues of block-modified Wigner matrix (histogram) and asymptotic distribution (solid line) computed with our method.}
\label{fig:Wigner}
\end{center}
\end{figure}

\begin{figure}[ht]
\begin{center}
\epsfig{file=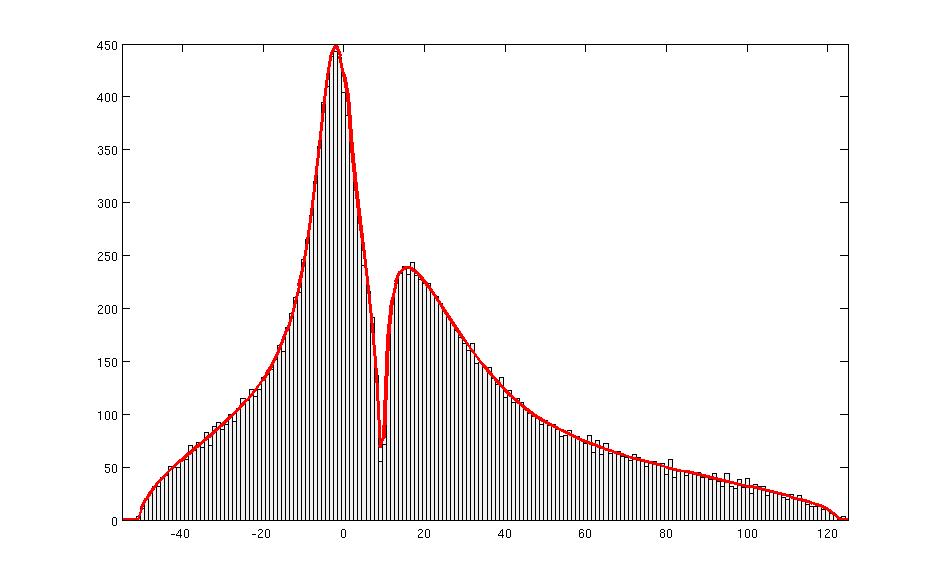, width=10cm}
\caption{Eigenvalues of block-modified Wishart matrix (histogram) and asymptotic distribution (solid line) computed with our method.}
\label{fig:Wishart}
\end{center}
\end{figure}

\begin{figure}[ht]
\begin{center}
\epsfig{file=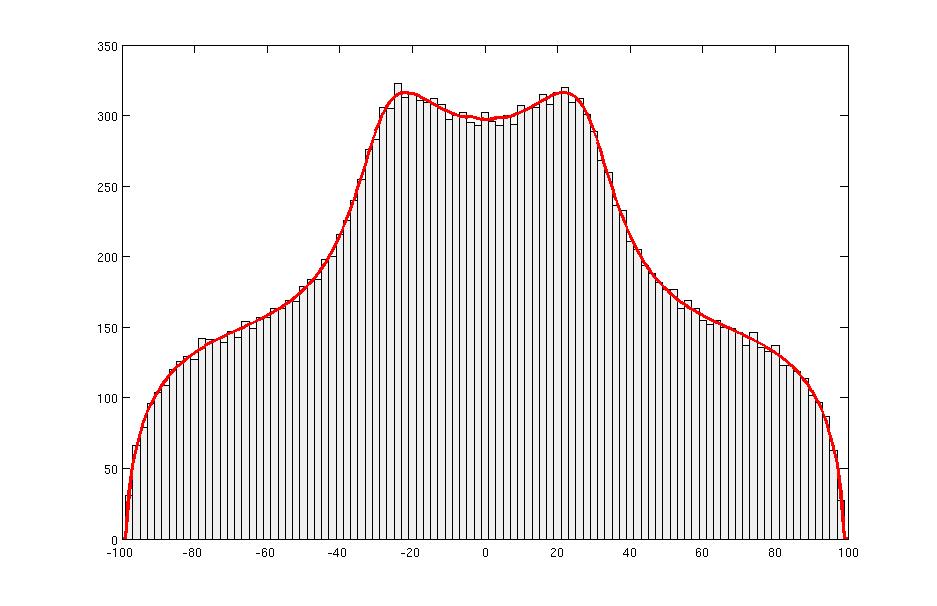, width=10cm}
\caption{Eigenvalues of block-modified randomly rotated arcsine matrix (histogram) and asymptotic distribution (solid line) computed with our method.}
\label{fig:arcsine}
\end{center}
\end{figure}

Now we want to obtain explicit formulas. For this, we will some conditions on $\varphi$ which allow us to express the asymptotic distribution of $X^{\varphi}$ as the product of operators which are free with amalgamation over a \textit{commutative algebra}.

\section{Freeness for block-modified random matrices}
\label{sec:freeness}
We shall denote by $\{E_i\}_{i=1}^r$ a fixed basis of the complex Hilbert space $\mathbb C^r$. Also, we put $E_{ij} = E_i^* E_j$. To any linear map $\varphi:M_m(\mathbb C) \to M_n(\mathbb C)$, one associates its Choi matrix \cite{choi1975completely}
\begin{equation}\label{eq:def-Choi}
M_n(\mathbb C) \otimes M_m(\mathbb C) \ni C = [\varphi \otimes \mathrm{id}](\Omega_m) = \sum_{i,j=1}^m \varphi(E_{ij}) \otimes E_{ij},
\end{equation}
where the Bell (or maximally entangled) matrix is $\Omega_m = \omega_m \omega_m^*$, with
\begin{equation}\label{eq:Bell}
\omega_m=\sum_{i=1}^m E_i \otimes E_i \in \mathbb C^m \otimes \mathbb C^m.
\end{equation}

As in Section \ref{sec:random-matrices}, we consider a sequence of self-adjoint unitarily invariant random matrices $X_d \in M_d(\mathbb C) \otimes M_m(\mathbb C)$ that converges in distribution towards some random variable $x \in \mathcal A$ having distribution $\mu$. Define
\begin{equation}\label{eq:def-c-phi}
X_d^\varphi = [\mathrm{id} \otimes \varphi](X) = \sum_{i,l=1}^n \sum_{j,k=1}^m c_{ijkl} (I_d \otimes E_{ij})X_d(I_d \otimes E_{kl}) \in M_d(\mathbb C) \otimes M_n(\mathbb C),
\end{equation}
for some coefficients $c_{ijkl} \in \mathbb C$, which are actually the entries of the Choi matrix of $\varphi$ (see Lemma \ref{lem:c-ijkl})
$$\forall \, 1 \leq i,l \leq n, \, 1  \leq j,k \leq m, \qquad c_{ijkl} = \langle E_{il} \otimes E_{jk}, C \rangle.$$

In order to consider the general case where $m \neq n$, we embed the matrices $X_d$, as well as the matrix units $E_{ij}, E_{kl}$ into the larger, square space $M_d(\mathbb C) \otimes M_{m+n}(\mathbb C)$, as follows:
\begin{align*}\hat X_d &= \begin{bmatrix}
X_d & 0_{m \times n} \otimes I_d \\
0_{n \times m} \otimes I_d & 0_{n \times n} \otimes I_d
\end{bmatrix}\\
\hat X^\varphi_d &=  \begin{bmatrix}
 0_{m \times m} \otimes I_d & 0_{m \times n} \otimes I_d \\
0_{n \times m} \otimes I_d & X^\varphi_d
\end{bmatrix} =  \sum_{i,l=1}^n \sum_{j,k=1}^m c_{ijkl}
(E_{m+i,j} \otimes I_d)
\cdot
\hat X_d
\cdot
(E_{k,m+l} \otimes I_d).
\end{align*}

Following Remark \ref{freemunits} (3), we have a nice description of the asymptotic distribution:

\begin{proposition}\label{prop:freeness-rectangular}
The sequence $\hat X^\varphi_d$ converges in distribution, as $d \to \infty$, to an element $\hat x^\varphi \in \mathcal A$ of the form
$$\hat x^\varphi = \sum_{i,l=1}^n \sum_{j,k=1}^m c_{ijkl} e_{m+i,j} \hat x e_{k,m+l},$$
where $\hat x$ and the family of matrix units $\{e_{ij}\}_{i,j=1}^{m+n}$ are free with amalgamation over the algebra $\langle p_m,p_n\rangle \subset \mathcal A$.

Here $p_m,p_n$ are orthogonal the projections $$p_m = \sum_{i=1}^m e_{ii} \quad \text{ and } \quad p_n = \sum_{i=m+1}^{m+n} e_{ii},$$ with traces $\tau(p_m) = m/(m+n)$, $\tau(p_n) = n/(m+n)$.

The distribution of $\hat x$ is $m/(m+n) \mu + n/(m+n) \delta_0$. In particular, the random matrix $X^\varphi_d$ converges in distribution, as $d \to \infty$, to the probability measure $\mu^\varphi$, which is such that the distribution of $\hat x^\varphi$ is $m/(m+n) \delta_0 + n/(m+n) \mu^\varphi$.
\end{proposition}

First, we exhibit the relation between the coefficients $c_{ijkl}$ and the Choi matrix $C$ of the linear map $\varphi$.

 \begin{lemma}\label{lem:c-ijkl}
 Consider the spectral decomposition of the Choi matrix \eqref{eq:def-Choi}
 $$C = \sum_{s=1}^{mn} \lambda_s v_s v_s^*=\sum_{t=1}^r\rho_t P_t,$$
 where the eigenvalues $\lambda_s \in \mathbb R$, are listed with multiplicity, $v_s \in \mathbb C^n \otimes \mathbb C^m$, $r$ is the rank of $C$, and $P_t, t\leq r$ is the projector to the eigenspace with non-zero eigenvalue $\rho_t$. Then, for all $i,j,k,l = 1, \ldots, m$,
 $$c_{ijkl} = \langle E_{il} \otimes E_{jk}, C \rangle = \sum_{s=1}^r \lambda_s \langle E_i \otimes E_j, v_s \rangle \overline{\langle E_l \otimes E_k, v_s \rangle}.$$
 \end{lemma}
\begin{proof}
Use equation \eqref{eq:def-c-phi} to find $c_{ijkl} = \langle E_{il}, \varphi(E_{jk}) \rangle$. Since $C = \sum_{ij} \varphi(E_{ij}) \otimes E_{ij}$, the first equality in the conclusion follows. The second one is obtained by using the spectral decomposition of the selfadjoint matrix $C$.
\end{proof}

\begin{proposition}
The block-modified random variable $\hat x^\varphi$ has the following expression in terms of the eigenvalues and of the eigenvectors of the Choi matrix $C$:
\begin{equation}\label{eq:x-phi-f}
\hat x^\varphi = f^* \hat x \otimes  C  f,
\end{equation}
where
$$ f = \sum_{s=1}^r   w_s^* \otimes  v_s \in \mathcal A \otimes M_{nm}(\mathbb C)$$
and the random variables $ w_s \in \mathcal A$ are defined by
\begin{equation}\label{eq:ws}
 w_s = \sum_{i=1}^n \sum_{j=1}^m \langle E_i \otimes E_j , v_s \rangle e_{m+i,j}.
 \end{equation}
\end{proposition}
\begin{proof}
The proof follows directly from Lemma \ref{lem:c-ijkl}. For the left hand side of \eqref{eq:x-phi-f}, we have
\begin{align*}
\hat x^\varphi &= \sum_{i,l=1}^n \sum_{j,k=1}^m c_{ijkl} e_{m+i,j}\hat xe_{k,m+l}\\
&=\sum_{i,l=1}^n \sum_{j,k=1}^m c_{ijkl} e_{m+i,j}xe_{m+l,k}^*\\
&=\sum_{i,l=1}^n \sum_{j,k=1}^m \sum_{s=1}^r \lambda_s \langle E_i \otimes E_j, v_s \rangle \overline{\langle E_l \otimes E_k, v_s \rangle} e_{m+i,j} \hat x e_{m+l,k}^*\\
&=\sum_{s=1}^r \lambda_s  w_s \hat x w_s^*.
\end{align*}
Using the orthogonality of the eigenvectors $v_s$, one can see easily that the right hand side of equation \eqref{eq:x-phi-f} gives the same expression as above, finishing the proof.
\end{proof}

For each element $\rho_t$ of the spectrum of $C$, let $J_t \subseteq \{1, \ldots nm\}$ be the set of eigenvectors $v_j$ appearing in the eigenprojector $P_t$:
 $$P_t=\sum_{j \in J_t} v_j v_j^*.$$
 We also define $Q_t:=\sum_{j \in J_t} w_j  w_j^* \in \mathcal A$.

 We are going to compute first the distribution of the random variable
 $$y = (x \otimes C) \cdot ff^*.$$
 The relation with the variable $x^\varphi$, in which we are ultimately interested in, is given by a normalization factor:
 $$\forall k \geq 1, \qquad \tau[(x^\varphi)^k] = nm \cdot \mathbb E [y^k].$$
 In particular, the same relation holds for the $\Psi$-transforms:
 \begin{equation}\label{eq:relation-Psi-y-xphi}
\Psi^{sc}_{x^\varphi} = r \Psi^{sc}_y.
\end{equation}

 \begin{definition} \label{def:C-well-behaved}
 We say that the eigenspaces of $C$ are \emph{tracially well behaved} if, for all $i_1,\dots,i_k\leq r$,
\begin{equation}\label{eq:cond-well-behaved}
\tau( w_{j_1}  w_{j_2}^*Q_{i_1}\dots Q_{i_k})=\delta_{j_1j_2}\tau( w_{j_1}  w_{j_1}^*Q_{i_1}\dots Q_{i_k})=\delta_{j_1j_2}\tau( w_{j_1'}  w_{j_1'}^*Q_{i_1}\dots Q_{i_k}),
\end{equation}
for every $j_1,j_1',j_2 \leq mn$ such that $j_1,j_1'\in J_i$ for some $i\leq r$.
\end{definition}

\begin{theorem}
Consider a linear map $\varphi:M_m(\mathbb C) \to M_n(\mathbb C)$ having a Choi matrix $C\in M_{nm}(\mathbb C) \subset \mathcal A \otimes M_{nm}(\mathbb C)$ which has tracially well behaved eigenspaces (see Definition \ref{def:C-well-behaved}). Then, the random variables $\hat x \otimes C$ and $ff^*$ are free with amalgamation over the (commutative) unital  algebra $\mathcal B = \langle p_m,p_n \rangle \otimes \langle  C \rangle$ generated by the projections $p_m,p_n$ from Proposition \ref{prop:freeness-rectangular} and the matrix $C$.
\end{theorem}
\begin{proof}
Clearly $\mathcal B$ is generated by $p_m,p_n$ and by the spectral projections $P_1,\ldots,P_r\in M_{nm}(\mathbb C) \subset \mathcal A \otimes M_{nm}(\mathbb C)$. Since $\hat x$ and $(e_{ij})_{ij}$ are free with amalgamation over $\langle p_m,p_n \rangle$ (see Proposition \ref{prop:freeness-rectangular}), it is well known that $\langle \hat x \rangle\otimes M_{nm}(\mathbb C),\langle (e_{ij})_{ij} \rangle\otimes M_{nm}(\mathbb C) \subset \mathcal A \otimes M_{nm}(\mathbb C)$ are $\langle p_m,p_n \rangle \otimes M_{nm}(\mathbb C)$-free, with respect to the conditional expectation $\mathbb E:=\psi \otimes \mathrm{id}$, where $\psi(a) = \tau(p_m a)p_m + \tau(p_n a)p_n$. Since $ff^*\in \langle (e_{ij})_{ij} \rangle\otimes M_{nm}(\mathbb C)$, $\hat x \otimes C$ and $ff^*$ are $\langle p_m,p_n \rangle \otimes M_{nm}(\mathbb C)$-free. By Proposition \ref{thm:Restrictions} and Corollary \ref{maintool}, to show $\mathcal B$-freeness, it is enough to show that
the $\langle p_m,p_n \rangle \otimes M_{nm}(\mathbb C)$-moments of $\hat x \otimes C$ and $ff^*$ restricted to $\mathcal B$ are also valued in $\mathcal B$.

The case of $\hat x \otimes C$ is immediate: for any $i_1, \ldots, i_{k-1} \in \{m,n\}$ and $j_1, \ldots, j_{k-1} \in \{1, \ldots, r\}$, we have
\begin{align*}
\mathbb E & \left[ (\hat x \otimes C)(p_{i_1} \otimes P_{j_1})(\hat x \otimes C)\cdots (\hat x \otimes C)(p_{i_{k-1}} \otimes P_{j_{k-1}})(\hat x \otimes C) \right] \\
& \qquad = \psi(\hat x p_{i_1} \hat x \cdots \hat x p_{i_{k-1}} \hat x) \otimes CP_{j_1}C \cdots C P_{j_{k-1}} C.
\end{align*}
Since $\hat x = p_m \hat x p_m$ and $p_mp_n = p_np_m=0$, the first factor of the tensor product is zero unless $i_1 = \cdots = i_{k-1} = m$, in which situation it is equal to $\tau(\hat x^k)p_m$. It is also obvious that the second factor is an element of $\langle C \rangle$, so the whole expression is $\mathcal B$-valued.

Let us now treat the case of $ff^*$. First, note that
$$ff^* = \sum_{h,h'=1}^r w_h^* w_{h'} \otimes v_h v_{h'}^*$$
and, for any $h,h',i$,
\begin{align*}
v_{h'}^* P_i v_h &= \delta_{h,h'} \mathbf{1}_{h \in J_i}\\
w_{h'} p_i w_h^* &= \mathbf{1}_{i=m} w_{h'} w_h^*.
\end{align*}
We develop
\begin{align}
 \mathbb E & \left[ ff^*(p_{i_1} \otimes P_{j_1})\cdots (p_{i_{k-1}} \otimes P_{j_{k-1}})ff^* \right] \\
& \qquad =\sum_{\substack{h_1,\dots,h_k=1\\h_1',\dots,h_k'=1}}^r\mathbb E \left[  w_{h_1}^* w_{h_1'} p_{i_1} \cdots p_{i_{k-1}} w_{h_k}^* w_{h_k'}\otimes v_{h_1} v_{h_1'}^*P_{j_1} v_{h_2} v_{h_2'}^* \cdots  P_{j_{k-1}} v_{h_k}v_{h_k'}^* \right]\\
\label{eq:cond-exp-prod-Q}&\qquad =\sum_{h_1,h_k'=1}^{r}\tau( w_{h_1}^*Q_{i_1}\cdots Q_{i_{k-1}} w_{h'_k}) p_n \otimes v_{h_1}v_{h'_k}^*\\
&\qquad =p_n \otimes \sum_{j=1}^rc_j  P_{j} \in \mathcal B,
\end{align}
for some constants $c_j$ and the assertion follows. In the last equality above, we have used the hypothesis satisfied by the $P_t$ (and thus by the $Q_t$, see Definition \ref{def:C-well-behaved})  and the tracial property of $\tau$.
\end{proof}

We would now like to investigate classes of Choi matrices $C$ which satisfy the condition in Definition \ref{def:C-well-behaved}.

We start with the following lemma:

\begin{lemma}\label{lem:covariance}
For any $s,t \in \{1, \ldots, r\}$,  we have $\tau(w_s w_t^*) = (m+n)^{-1}\delta_{st}$.
\end{lemma}
\begin{proof}
Using equation \eqref{eq:ws}, we have
\begin{align*}
\tau(w_s w_t^*) &= \sum_{i1,i_2=1}^n \sum_{j_1,j_2=1}^m \langle E_{i_1} \otimes E_{j_1} , v_s \rangle  \overline{\langle E_{i_2} \otimes E_{j_2} , v_t \rangle}
 \tau(e_{m+i_1,j_1} e_{j_2,m+i_2}) \\
 &= \sum_{i1,i_2=1}^n \sum_{j_1,j_2=1}^m \langle E_{i_1} \otimes E_{j_1} , v_s \rangle  \overline{\langle E_{i_2} \otimes E_{j_2} , v_t \rangle}
 \delta_{i_1i_2} \delta_{j_1 j_2} (m+n)^{-1} \\
 &= (m+n)^{-1}\langle v_t, v_s \rangle = (m+n)^{-1}\delta_{st}.
 \end{align*}
\end{proof}

\begin{definition}\label{def:UC}
The Choi matrix $C$ is said to satisfy the \emph{unitarity condition} (UC) if, for all $t$, there is some real constant $d_t$ such that $Q_t = d_t \cdot 1$.
\end{definition}

Many Choi matrices associated to linear maps of practical importance satisfy the unitary condition (UC), see, for example, Sections \ref{sec:partial-transposition} and \ref{sec:partial-reduction}.

\begin{proposition}\label{prop:UC}
Any Choi matrix $C$ satisfying the unitarity condition (UC) satisfies also the condition from Definition \ref{def:C-well-behaved}. Moreover, $d_t =n^{-1} \mathrm{rk} P_t$.
\end{proposition}
\begin{proof}
Using (UC) and Lemma \ref{lem:covariance}, we have
\begin{equation}\label{eq:UC}
\tau( w_{j_1}  w_{j_2}^*Q_{i_1}\dots Q_{i_k})=d_{i_1}d_{i_2} \cdots d_{i_k} \tau(w_{j_1}  w_{j_2}^*) =n^{-1} d_{i_1}d_{i_2} \cdots d_{i_k} \delta_{j_1j_2},
\end{equation}
which is enough to prove the statement in \eqref{eq:cond-well-behaved}. For the second claim, note that
$$d_t = \tau(Q_t) = \sum_{j \in J_t} \tau(w_jw_j^*) =n^{-1} \sum_{j \in J_t} \|v_j\|^2 =n^{-1} \mathrm{Tr} P_t.$$
\end{proof}

\begin{remark}\label{WB-UC}
There are maps which satisfy the condition in Definition \ref{def:C-well-behaved} but not the condition (UC). For example, the map $\varphi:M_n(\CC)\to M_m(\CC)$, given by:
$$\varphi(A)= \sum_{i\leq n, j\leq m} \alpha_{ij}E_{ij}AE_{ji},$$
has a diagonal Choi matrix $C=\mathrm{diag}(\alpha_{11},\alpha_{21},\dots, \alpha_{n1},\alpha_{12}, \dots, \alpha_{nm})$.
Hence the eigenvectors of $C$ are just the standard basis elements $(E_j\otimes E_k)_{j,k\leq n}$ of $\CC^{n^2}$.

Therefore, the ``matrization'' $w_{jk}$ of the eigenvector $E_j\otimes E_k$ is simply the matrix unit $e_{jk}$. Thus, $Q_{ij}=e_{ii}$ are projections and $\varphi$ is does not satisfy (UC). However, we have that
\begin{equation}\label{usingWB}
\tau( w_{i_1j_1}  w_{i_2j_2}^*Q_{i_3j_3}\dots Q_{i_kj_k})=\tau( e_{i_1j_1}e_{j_2i_2}^*e_{i_3i_3}\dots e_{i_ki_k})=n^{-1}\delta_{j_1j_2}\prod_{s<k}\delta_{i_si_{s+1}},
\end{equation}
which implies \ref{eq:cond-well-behaved}. 

By the discussion at the end of Section \ref{sec:random-matrices} we can actually compute the distribution of $x^{\varphi}=\sum_{i\leq m} c_{ii}\otimes x_{ii}$. Indeed, since the $(x_{ii})$ are free in $\cA_{11}$, the distribution $\mu$ of $x^{\varphi}$ is just a convex combination of probability mesures $\mu=\frac{1}{n}(\mu_1+\dots+\mu_n)$, where each $\mu_i$ is the distribution of a linear combination $\alpha_{i1}x_{11}+\dots+\alpha_{im}x_{mm}$, where the $(x_{jj})$ are free, identically distributed random variables, each having the same distribution as the free compression $x_{11}$ of $x$.
\end{remark}

We consider now the dual $\varphi^*$ of a linear map $\varphi : M_m(\mathbb C) \to M_n(\mathbb C)$ acting on the blocks of a unitarily invariant random matrix. Recall that $\varphi^* : M_n(\mathbb C) \to M_m(\mathbb C)$ is defined by the duality relation (with respect to the Euclidean, or Hilbert-Schmidt, scalar product)
$$\forall X \in M_n(\mathbb C), Y \in M_m(\mathbb C), \quad \langle X, \varphi(Y) \rangle = \langle \varphi^*(X), Y \rangle.$$

The Choi matrices of $\varphi$ and $\varphi^*$ are related in a simple way, see \cite{FuNe15} and their eigenprojectors share similar properties.

\begin{lemma}
Let $\varphi : M_m(\mathbb C) \to M_n(\mathbb C)$ be a linear map and $\varphi^* : M_n(\mathbb C) \to M_m(\mathbb C)$ be its dual with respect to the Euclidean scalar product. Then,
$$C_{\varphi^*} = F_{n,m} C_\varphi^\top F_{n,m}^*,$$
where $F_{n,m}$ is the unitary operator acting on $\mathbb C^n \otimes \mathbb C^m$ as $F_{n,m} x \otimes y = y \otimes x$ for $x \in \mathbb C^n$ and $y \in \mathbb C^m$. In particular, if $\varphi$ is hermiticity preserving, then $C_\varphi$ and $C_{\varphi^*}$ have the same spectrum.\end{lemma}
\begin{proof}
The first statement is contained in \cite[Lemma 4.3]{FuNe15}. The second statement follows from the fact that the eigenvectors of $C_{\varphi^*}$ are obtained from the eigenvectors of $C_\varphi$ by coordinate-wise complex conjugation and swapping of the tensor factors.
\end{proof}

\section{Distribution of the product}
\label{sec:distribution}

In the previous section, we have shown that freeness with amalgamation holds for our model, under some fairly general assumptions on the Choi matrix $C$. In this section, we use the freeness result to compute explicitly the distribution of the modified element $x^\varphi$; we do however require some extra assumptions on the Choi matrix, in order to obtained a precise characterization of the modified distribution. If for the freeness result, the assumption that the eigenvectors $C$ are well-behaved (see Definition \ref{def:C-well-behaved}) is enough, in the next result we require the stronger unitarity condition from Definition \ref{def:UC}.

\begin{theorem}\label{thm:distribution-modified}
If the Choi matrix $C$ satisfies the unitarity condition from Definition \ref{def:UC}, then the distribution of $x^\varphi$ has the following $R$-transform:
\begin{equation}\label{eq:R-transf-x-phi}
 R_{x^\varphi}(z) = \sum_{i=1}^s d_i \rho_i R_x\left[ \frac{\rho_i}{n} z\right],
\end{equation}
where $\rho_i$ are the distinct eigenvalues of $C$ and $n d_i$ are ranks of the corresponding eigenprojectors. In other words, if $\mu$, resp.~ $\mu^\varphi$, are the respective distributions of $x$ and $x^\varphi$, then
\begin{equation}\label{eq:mu-x-phi}
\mu^\varphi = \boxplus_{i=1}^s (D_{\rho_i/n} \mu)^{\boxplus nd_i}.
\end{equation}
\end{theorem}
\begin{proof}
We start by computing the $\mathcal B$-valued $\Psi$-transforms of the elements $x \otimes C$ and $ff^*$. For simplicity, we shall see the algebra $\mathcal B$ as a commutative subalgebra of $M_{nm}(\mathbb C)$ by identifying the elements $1 \otimes B$ and $B$, for any $B \in M_{nm}(\mathbb C)$. We have
$$\Psi_{x \otimes C}(B) = \sum_{k \geq 1} \mathbb E \left[ (x \otimes C \cdot 1 \otimes B)^k \right] = \sum_{k \geq 1} \tau(x^k) (CB)^k = \Psi_x(CB),$$
where $\Psi_x(\cdot)$ is extended on $\mathcal B  = \langle C \rangle$ by (polynomial) functional calculus. The inverse $\Psi$ transform and the $S$-transform are easily deduced from the above expression:
$$\Psi_{x \otimes C}^{< -1 >}(B) = C^{-1} \Psi_{x}^{< -1 >}(B), \qquad S_{x \otimes C}(B) = C^{-1} S_{x}(B)$$
where $C^{-1}$ is the pseudo-inverse of $C$,
$$C^{-1} = \sum_{t:\rho_t \neq 0} \rho_t^{-1} P_t.$$

Let us now move on to the case of $ff^*$. For a general element $B = \sum_{t=1}^s b_t P_t$, we have
\begin{align*}
\Psi_{ff^*}(B) &= \sum_{k \geq 1} \mathbb E \left[ \left(ff^* \cdot \sum_{t=1}^s b_t 1 \otimes P_t \right)^k \right] \\
&= \sum_{k \geq 1} \sum_{t_1, \ldots, t_{k-1}=1}^s b_{t_1} \cdots b_{t_{k-1}} \mathbb E \left[ ff^* (1\otimes P_{t_1}) ff^* \cdots ff^* (1\otimes P_{t_{k-1}}) ff^* \right] \cdot B.
\end{align*}
Using the unitarity condition (UC) and the relation \eqref{eq:UC}, we get
\begin{align}
\Psi_{ff^*}(B) &= \sum_{k \geq 1} \sum_{t_1, \ldots, t_{k-1}=1}^s b_{t_1} \cdots b_{t_{k-1}} \mathbb E \left[ ff^* (1\otimes P_{t_1}) ff^* \cdots ff^* (1\otimes P_{t_{k-1}}) ff^* \right] \cdot B \\
&= \sum_{k \geq 1} \sum_{t_1, \ldots, t_{k-1}=1}^s b_{t_1} \cdots b_{t_{k-1}} \left( \sum_{h=1}^r  d_{t_1} \cdots d_{t_{k-1}} v_h v_h^* \right) n^{-1}B  \label{usingUC}\\
&= \left[ \sum_{k \geq 1} \left( \sum_{t=1}^s b_t d_t \right)^{k-1} \right]n^{-1} B\\
&= \frac{n^{-1}B}{1-n^{-1}\mathrm{Tr} B} =  \frac{B}{n-\mathrm{Tr} B},
\end{align}
where we have used the fact that $d_t = n^{-1}\mathrm{Tr} P_t$, see Proposition \ref{prop:UC} (the above relation holds for small enough $B$).
The inverse $\Psi$-transform and the $S$-transform follow easily from the relations above:
$$\Psi_{ff^*}^{<-1>}(B) = \frac{nB}{1+ \mathrm{Tr} B}, \qquad S_{ff^*}(B) = n\frac{B+I}{1+ \mathrm{Tr}B}.$$

Since the elements $x \otimes C$ and $ff^*$ are free with amalgamation over a commutative algebra, the $S$-transform of the product $y $ and the inverse $\Psi$-transform are given by
\begin{align}
\nonumber S_{y}(B) &= S_{x \otimes C}(B) S_{ff^*}(B) = nC^{-1} S_x(B)\frac{B+I}{1+ \mathrm{Tr}B} \\
 \label{eq:Psi-inverse}\Psi_{y}^{<-1>}(B) &= nC^{-1}\frac{S_x(B) B}{1+ \mathrm{Tr}B}.
 \end{align}

Since we are only interested in the scalar distribution of the random variable $y$, we only need to invert the function $ \Psi_{y}^{<-1>}$ on scalar elements, i.e. we need to solve the equation $ \Psi_{y}^{<-1>}(B) =a I$, for some real value $a$, and then we obtain $\Psi^{sc}_{y}(a) = \mathrm{tr}(B)$. In other words, using \eqref{eq:relation-Psi-y-xphi}, we have $\Psi^{sc}_{x^\varphi}(a) = \mathrm{Tr}(B)$. Let $B = \sum_{i=1}^s b_i P_i$ be a solution of the equation $ \Psi_{y}^{<-1>}(B) =a I$. On the eigenprojector $P_i$, this equation reads
$$\frac{n}{\rho_i} \cdot \frac{b_iS^{sc}_x(b_i)}{1+\mathrm{Tr} B} = a.$$
Let now $z_i$ be such $b_i = z_i R^{sc}_x(z_i)$. Using the formula $zR(z) + S(zR(z)) = z$ and the previous equation, we obtain, for all $i$,
$$z_i = \frac{\rho_i}{n} a(1 + \mathrm{Tr} B).$$
Plugging this value into $\Psi^{sc}_{x^\varphi}(a) = \mathrm{Tr}(B)$, we get
$$\Psi^{sc}_{x^\varphi}(a) = \mathrm{Tr}(B) = \sum_{i=1}^s nd_i b_i = \sum_{i=1}^s nd_i z_i R_x^{sc}(z_i) = \sum_{i=1}^s d_i \rho_i a(1 + \mathrm{Tr} B) R_x^{sc}\left[ \frac{\rho_i}{n} a(1 + \mathrm{Tr} B)\right].$$
Replacing $a$ by $1/(R^{sc}_{x^\varphi}(\alpha) + 1/\alpha)$ and using $\mathrm{Tr}(B) = \Psi^{sc}_{x^\varphi}(a) = \alpha R^{sc}_{x^\varphi}(\alpha)$, we obtain $a(1 + \mathrm{Tr} B) = \alpha$. Thus, we finally get
$$ \alpha R^{sc}_{x^\varphi}(\alpha) = \sum_{i=1}^s d_i \rho_i \alpha R^{sc}_x\left[ \frac{\rho_i}{n} \alpha\right],$$
which is our conclusion.
\end{proof}

\begin{remark}
From the previous theorem, it follows that whenever the measure $\mu$ is supported on the positive half-line, i.e~ $\mu([0, \infty)) = 1$, and the Choi matrix $C$ is positive semi-definite, the modified measure $\mu^\varphi$ is also supported on the positive reals. This is a consequence of a linear algebra fact: $C$ being positive semi-definite, the map $\varphi$ is \emph{completely positive}, and thus, for any positive semi-definite matrix $X$, the block-modified matrix $X^\varphi = [\mathrm{id} \otimes \varphi](X)$ is also positive semi-definite. However, neither converse holds: there exist matrices $C$ with negative eigenvalues such that $\mu^\varphi$ is not positively supported, although $\mu$ is (see Section \ref{sec:partial-transposition}), and there exist measures $\mu$ which are not positively supported such that $\mu^\varphi$ is positively supported (see the recent work \cite{CoHaNe15}).
\end{remark}

\section{Examples: specific distributions}
\label{sec:examples-distributions}

\subsection{GUE matrices}

Let us consider in this subsection the case of GUE matrices converging to a semicircular measure $\mu_{sc}$ of mean $a$ and variance $\sigma^2$ (see Section \ref{sec:random-matrices} for the definitions). This measure is characterized by the fact that only its first two free cumulants are non-zero:
\begin{equation}\label{eq:R-semicircular}
R_{\mu_{sc}}(z) = a + \sigma^2 z.
\end{equation}

\begin{proposition}\label{prop:GUE}
Consider a sequence of GUE matrices $X_d \in M_d(\mathbb C) \otimes M_m(\mathbb C)$ converging in distribution to the semicircular measure $\mu{sc}$ of mean $a \in \mathbb R$ and variance $\sigma^2 \geq 0$. Then, for any linear map $\varphi: M_m(\mathbb C) \to M_n(\mathbb C)$ satisfying the unitarity condition (UC) from Definition \ref{def:UC}, the sequence of modified matrices $X^\varphi = [\mathrm{id} \otimes \varphi](X) \in M_d(\mathbb C) \otimes M_n(\mathbb C)$ converge in distribution to another semicircular measure of mean $a^\varphi$ and variance $(\sigma^\varphi)^2$, where
\begin{align*}
a^\varphi &=  \frac{\mathrm{Tr}C}{n} a\\
\sigma^\varphi &= \frac{\sigma}{n} \sqrt{\mathrm{Tr}(C^2)}.
\end{align*}
\end{proposition}
\begin{proof}
The proof follows from Theorem \ref{thm:distribution-modified} and the form of the $R$-transform for semicircular random variables \eqref{eq:R-semicircular}.
\end{proof}

\subsection{Compound Wishart matrices}

Compound free Poisson distributions were introduced in \cite{Sp98} as a generalization of the free Poisson (or Mar\u{c}enko-Pastur) distribution. Given a positive measure $\mu$ on the real line, the \emph{compound free Poisson} distribution of parameter $\mu$ is the unique ($\boxplus$-infinitely divisible) probability measure on the real line having the moments of $\mu$ as its free cumulants, see Section \ref{sec:random-matrices} and \cite[Proposition 3.3.11]{HiPe00}.

As in the case of the usual free Poisson distribution, there is a very natural random matrix model for the compound free Poisson class of measures. Let $\lambda$ be the mass of $\mu$ and consider the normalized probability measure $\tilde \mu(\cdot) = \mu(\cdot) / \lambda$. Let $X_d \in M_{d \times s_d}(\mathbb C)$ be a sequence of random Ginibre matrices (i.e.~ random matrices with i.i.d.~ standard, complex Gaussian entries) such that $s_d \sim \lambda d$. Consider also a sequence of selfadjoint matrices $D_d \in M_{s_d}(\mathbb C)$, independent from $X_d$ (or deterministic) that converge in distribution to $\tilde \mu$. Then, the sequence of random matrices
$$\frac{1}{s_d} X_d D_d X_d^*$$
converges in distribution to the compound free Poisson distribution $\pi_\mu$; see \cite[Proposition 4.4.11]{HiPe00}.

In the next proposition, we show that the class of compound free Poisson distributions is stable under linear block modifications.

\begin{proposition}\label{prop:Wishart}
Consider a sequence of compound Wishart matrices $X_d \in M_d(\mathbb C) \otimes M_m(\mathbb C)$ converging in distribution to a compound free Poisson measure $\pi_\mu$, where $\mu$ is a positive measure supported on the real line (not necessarily of unit mass). Then, for any linear map $\varphi: M_m(\mathbb C) \to M_n(\mathbb C)$ satisfying the unitarity condition (UC) from Definition \ref{def:UC}, the sequence of modified matrices $X^\varphi = [\mathrm{id} \otimes \varphi](X) \in M_d(\mathbb C) \otimes M_n(\mathbb C)$ converge in distribution to another compound free Poisson measure $\pi_{\mu^\varphi}$, where
$$\mu^\varphi = \sum_{i=1}^s nd_i D_{\rho_i/n}[\mu].$$

In particular, if the limiting eigenvalue distribution of the random matrices $X_d$ is a (usual) free Poisson measure of parameter $\lambda>0$ (i.e.~ $\mu = \lambda \delta_1$), then the limiting eigenvalue distribution of the modified matrices $X_d^\varphi$ is a compound free Poisson distribution of parameter
$$\mu^\varphi = \lambda\sum_{i=1}^s  nd_i \delta_{\rho_i/n},$$
which is $\lambda m n $ times the empirical eigenvalue distribution of the rescaled Choi matrix $C/n$.
\end{proposition}
\begin{proof}
Using Theorem \ref{thm:distribution-modified} and the fact that the free cumulants of a compound free Poisson measure $\pi_\mu$ are the moments of $\mu$, we have ($m_p$ denotes the $p$-th moment of a measure)
\begin{align*}
R_{x^\varphi}(z) &= \sum_{i=1}^s d_i \rho_i R_x\left[ \frac{\rho_i}{n} z\right] =\sum_{i=1}^s d_i \rho_i R_{\pi_\mu}\left[ \frac{\rho_i}{n} z\right] \\
&=  \sum_{i=1}^s d_i \rho_i \sum_{p=0}^\infty m_{p+1}(\mu) \left(\frac{\rho_i}{n} z\right)^p \\
&=  \sum_{i=1}^s d_i \rho_i \frac{n}{\rho_i} \sum_{p=0}^\infty m_{p+1}(D_{\rho_i/n}[\mu]) z^p \\
&=  \sum_{p=0}^\infty m_{p+1} \left( \sum_{i=1}^s nd_i D_{\rho_i/n}[\mu]\right) z^p\\
 &= R_{\pi_{\mu^\varphi}}(z).
\end{align*}

\end{proof}

\begin{remark}
The above result generalizes \cite[Theorem 3.1]{BaNe12b}, in the case where the assumptions of both statements are satisfied. Note however that the unitarity condition (UC) from Definition \ref{def:UC} is, in practice, much easier to verify that the set of conditions from \cite[Theorem 3.1]{BaNe12b}, which are indexed by an integer parameter $p \in \mathbb N$.
\end{remark}

\section{Examples: specific linear maps}
\label{sec:examples-maps}

\subsection{Unitary rotations of arbitrary random matrices}

To warm up, let us start with the unitary rotation map $\varphi:M_n(\mathbb C) \to M_n(\mathbb C)$ given by $\varphi(X) =UXU^*$, for some fixed unitary rotation $U \in \mathcal U_n$. In this case, it is clear that for any (random) bi-partite matrix $M_d(\mathbb C) \otimes M_n(\mathbb C)$, one has
$$X^\varphi = [\mathrm{id} \otimes \varphi](X) = (I \otimes U) X (I \otimes U)^*,$$
so that $X^\varphi$ and $X$ have the same spectrum.

Applying the formalism of our work to this case, we find that the Choi matrix corresponding to map $\varphi$ has unit rank: $C=uu^*$, where $u \in \mathbb C^n \otimes \mathbb C^n$ is the ``vectorization'' of the square matrix $U$:
$$u = \sum_{i,j=1}^n \langle e_i , U e_j \rangle e_i \otimes e_j = \sum_{i,j=1}^n U_{ij} e_i \otimes e_j.$$
Since $\|u\| = \sqrt n$, the Choi matrix $C$ has a unique non-zero eigenvalue, equal to $n$, of multiplicity one. Notice that $C$ satisfies the unitarity condition (UC) from Definition \ref{def:UC}, since this property is equivalent to the unitarity of the matrix $U$. Considering a sequence of matrices $X_d$, converging in distribution to a limit $\mu$, we apply Theorem \ref{thm:distribution-modified}, verifying that $\mu^\varphi = \mu$.

\subsection{The trace and its dual}

In this subsection, we look at the following two maps:
\begin{align*}
\mathrm{\psi_A}: M_m(\mathbb C) \to \mathbb C \cong M_1(\mathbb C),\quad &\psi_A(X) = \mathrm{Tr}(A^\top X) \\
J : M_1(\mathbb C) \to M_n(\mathbb C), \quad &J(z) = xI_n,
\end{align*}
where $A \in M_m(\mathbb C)$ is a self-adjoint matrix and $x \in \mathbb R$. Note that, for $m=n$, the maps $J$ and $\psi_I$ are dual to one other. The Choi matrices read, respectively,
\begin{align*}
C_{\psi_A} &= 1 \otimes A \in M_1(\mathbb C) \otimes M_m(\mathbb C)\\
C_J &= xI_n \otimes 1 \in M_n(\mathbb C) \otimes M_1(\mathbb C),
\end{align*}
and it is easy to check that they both satisfy the unitarity condition (UC) from Definition \ref{def:UC}. In the case of the map $J$, the modified matrix is $X_d^J = zX_d \otimes I_n$, so its spectral distribution will converge to $D_x[\mu]$, as predicted by Theorem \ref{thm:distribution-modified}. Using the same theorem, we can prove the following proposition.

\begin{proposition}\label{prop:trace}
Consider a sequence of unitarily invariant random matrices $X_d \in M_d(\mathbb C) \otimes M_m(\mathbb C)$ converging in distribution to a probability measure $\mu$. Then, the sequence of  matrices $X_d^{\psi_A}:= [\mathrm{id} \otimes \psi_A](X_d) \in M_d(\mathbb C)$ converges in distribution to the probability measure $\mu^{\psi_A}$ given by
\begin{equation}
\mu^{\psi_A} = D_{\lambda_1}[\mu] \boxplus D_{\lambda_2}[\mu] \boxplus \cdots \boxplus D_{\lambda_m}[\mu],
\end{equation}
where $\lambda_1, \ldots , \lambda_m$ are the eigenvalues of $A$. In particular, when $A = I_m$, $X_d^{\psi_I}:= [\mathrm{id} \otimes \mathrm{Tr}](X_d) \in M_d(\mathbb C)$ is the partial trace of $X_d$, and
\begin{equation}
\mu^{\psi_I} = \mu^{\boxplus m}.
\end{equation}
\end{proposition}

\subsection{The partial transposition}
\label{sec:partial-transposition}

Historically, the study linear modifications of block-random matrices was initiated by Aubrun in \cite{Au12}, where he studied partial transpositions of large Wishart random matrices. The same random matrix model, in the asymptotical regime where one dimension is kept fixed, while the other grows, was analyzed in \cite{BaNe12}, and later generalized in \cite{BaNe12b}.

Let us consider the transpose map, $\mathrm{transp}: M_n(\mathbb C) \to M_n(\mathbb C)$, $\mathrm{transp}(X) = X^\top$. Its Choi matrix is given by the \emph{flip operator} $F \in M_n(\mathbb C) \otimes M_n(\mathbb C)$, which is the unitary matrix defined by $F(x \otimes y) = y \otimes x$ for all $x,y \in \mathbb C^n$. The spectral decomposition of the flip operator reads $F = P_+-P_-$, where $P_+$, resp.~ $P_-$, are the orthogonal projections on the symmetric, resp.~ anti-symmetric subspaces of the tensor product $\mathbb C^n \otimes \mathbb C^n$. Note that the transposition map satisfies the unitarity condition (UC) from Definition \ref{def:UC}; this is a consequence of the relations
$$P_+ = \frac{I + F}{2} \quad \text{ and } \quad P_- = \frac{I - F}{2},$$
and of the fact that both the identity and the flip have partial traces which are proportional to the identity.

\begin{proposition}\label{prop:partial-transposition}
Consider a sequence of unitarily invariant random matrices $X_d \in M_d(\mathbb C) \otimes M_n(\mathbb C)$ converging in distribution to a probability measure $\mu$. Then, the sequence of partially transposed matrices $X_d^\Gamma:= [\mathrm{id} \otimes \mathrm{transp}](X_d) \in M_d(\mathbb C) \otimes M_n(\mathbb C)$ converges in distribution to the probability measure $\mu^\Gamma$ given by
\begin{equation}
\mu^\Gamma = D_{1/n} \left[ \mu^{\boxplus n(n+1)/2} \boxminus  \mu^{\boxplus n(n-1)/2} \right],
\end{equation}
where $\boxminus$ is denotes following operation: $\nu_1 \boxminus \nu_2 := \nu_1 \boxplus D_{-1}\nu_2$.
\end{proposition}
\begin{proof}
Since the eigenvalues of the Choi matrix $F$ of the transposition map are $+1$, $-1$, with respective multiplicities $nd_1 = n(n+1)/2$, $nd_2 = n(n-1)/2$, we have from Theorem \ref{thm:distribution-modified}
\begin{align*}
\mu^\Gamma &= (D_{1/n}\mu)^{\boxplus n(n+1)/2} \boxplus (D_{-1/n}\mu)^{\boxplus n(n+1)/2} \\
&= D_{1/n} \left[ \mu^{\boxplus n(n+1)/2} \boxplus (D_{-1}\mu)^{\boxplus n(n+1)/2} \right] \\
&=D_{1/n} \left[ \mu^{\boxplus n(n+1)/2} \boxminus  \mu^{\boxplus n(n-1)/2} \right].
\end{align*}
\end{proof}

In the case where the measure $\mu$ is a free Poisson measure of parameter $\lambda>0$, we recover Theorem 4.1 from \cite{BaNe12}:
\begin{corollary}
Let $W_d \in M_d(\mathbb C) \otimes M_n(\mathbb C)$ be a sequence of Wishart random matrices converging in distribution to a free Poisson distribution of parameter $\lambda$, $\pi_\lambda$. Then, the partially transposed random matrices $W_d^\Gamma =  [\mathrm{id} \otimes \mathrm{transp}](W_d)$ converge in distribution to a rescaling of a free difference of free Poisson distributions of respective parameters $\lambda n(n+1)/2$ and $\lambda n(n-1)/2$:
$$\pi_\lambda^\Gamma = D_{1/n} \left[ \pi_{\lambda n(n+1)/2} \boxminus  \pi_{\lambda n(n-1)/2}\right].$$
\end{corollary}

For random projections, we recover a result proved in \cite[Proposition 6.5]{FuNe15}:
\begin{corollary}
Let $P_d \in M_d(\mathbb C) \otimes M_n(\mathbb C)$ be a sequence of random projection operators converging in distribution to a Bernoulli distribution $b_t = (1-t) \delta_0 + t\delta_1$, for some fixed parameter $t \in (0,1)$. Then, the partially transposed random matrices $P_d^\Gamma =  [\mathrm{id} \otimes \mathrm{transp}](P_d)$ converge in distribution to a rescaling of a free difference of free additive powers of Bernoulli distributions:
$$\pi_\lambda^\Gamma = D_{1/n} \left[ b_t^{\boxplus n(n+1)/2} \boxminus b_t^{n(n-1)/2}\right].$$
\end{corollary}

\subsection{The reduction map}
\label{sec:partial-reduction}

In \cite{JLN14,JLN15}, the authors study the \emph{reduction map}, following a motivation originating in quantum information theory. The reduction map is another example of positive, but not completely positive application between matrix algebra, and its study has originated with the introduction of the reduction criterion \cite{CAG99}. It is of particularly simple form: for $X \in M_n(\mathbb C)$,
$$X^{red}:= \mathrm{red}(X) = I_n \cdot \mathrm{Tr}(X) - X.$$

\begin{proposition}\label{prop:partial-reduction}
Consider a sequence of random matrices $X_d \in M_d(\mathbb C) \otimes M_n(\mathbb C)$ converging in distribution to ta probability measure $\mu$. Then, the sequence of partially reduced matrices $X_d^{red}:= [\mathrm{id} \otimes \mathrm{red}](X_d) \in M_d(\mathbb C) \otimes M_n(\mathbb C)$ converges in distribution to the probability measure $\mu^{red}$ given by
\begin{equation}
\mu^{red} = D_{1/n} \left[ \mu^{\boxplus n^2-1} \boxminus  D_{n-1}\mu \right].
\end{equation}
\end{proposition}
\begin{proof}
The Choi matrix of the reduction map $\mathrm{red}$ reads $C_{red} = I-F$. Its eigenvalues are respectively $1$ and $1-n$, with multiplicities $n^2-1$, resp.~ $1$. Moreover, the eigenvector associated to the eigenvalue $1-n$ is a multiple of the Bell state \eqref{eq:Bell}, having a partial trace equal to the identity operator on $\mathbb C^n$. Hence, the Choi matrix satisfies the unitarity condition (UC) and, using Theorem \ref{thm:distribution-modified}, we have
\begin{align*}
\mu^{red} &= (D_{1/n}\mu)^{\boxplus n^2-1} \boxplus (D_{(1-n)/n}\mu)\\
&= D_{1/n} \left[ \mu^{\boxplus n^2-1} \boxplus D_{1-n}\mu\right] \\
&= D_{1/n} \left[ \mu^{\boxplus n^2-1} \boxminus  D_{n-1}\mu\right].
\end{align*}
\end{proof}

In the case where the measure $\mu$ is a free Poisson measure of parameter $\lambda>0$, we recover \cite[Theorem 9.1]{JLN14}:
\begin{corollary}
Let $W_d \in M_d(\mathbb C) \otimes M_n(\mathbb C)$ be a sequence of Wishart random matrices converging in distribution to a free Poisson distribution of parameter $\lambda$, $\pi_\lambda$. Then, the partially reduced random matrices $W_d^{red} =  [\mathrm{id} \otimes \mathrm{red}](W_d)$ converge in distribution to a rescaled compound free Poisson distribution $D_{1/n} \pi_{\mu_{n,\lambda}}$, where
$$\mu_{n,\lambda} =  \lambda(n^2-1)\delta_1 + \lambda \delta_{1-n}.$$
\end{corollary}

\subsection{Generalized partial transposition and reduction maps}

We generalize in this subsection the results from Sections \ref{sec:partial-transposition} and \ref{sec:partial-reduction}, by considering the map $\varphi: M_n(\mathbb C) \to M_n(\mathbb C)$ defined by
\begin{equation}\label{eq:partial-generalization}
\varphi(A) = \alpha A + \beta \mathrm{Tr}(A) I_n  + \gamma A^\top.
\end{equation}
Its Choi matrix reads
$$C_\varphi = \alpha \Omega_n  + \beta I_{n^2} + \gamma F.$$
The eigenvalues and eigenprojectors of $C_\varphi$ are as follows:
\begin{align*}
\lambda_1 = n \alpha + \beta + \gamma &\qquad P_1 = n^{-1}\Omega_n\\
\lambda_2 =  \beta + \gamma &\qquad P_2 = \frac{I+F}{2} - n^{-1}\Omega_n\\
\lambda_3 = \beta - \gamma &\qquad P_3 = \frac{I-F}{2}.
\end{align*}
Since all the partial traces of the eigenprojectors $P_{1,2,3}$ are multiples of the identity operator, the Choi matrix $C_\varphi$ satisfies the unitarity condition (UC), and the following results hold (we leave the proofs to the reader).

\begin{proposition}\label{prop:partial-generalization}
Consider a sequence of random matrices $X_d \in M_d(\mathbb C) \otimes M_n(\mathbb C)$ converging in distribution to ta probability measure $\mu$. Then, for the application $\varphi$ from \eqref{eq:partial-generalization}, the sequence of the block modified matrices $X_d^{\varphi}:= [\mathrm{id} \otimes \varphi](X_d) \in M_d(\mathbb C) \otimes M_n(\mathbb C)$ converges in distribution to the probability measure $\mu^{\varphi}$ given by
\begin{equation}
\mu^{\varphi} = D_{\alpha + (\beta + \gamma)/n} \left[ \mu \right] \boxplus  D_{(\beta + \gamma)/n} \left[ \mu^{\boxplus n(n+1)/2-1} \right] \boxplus  D_{(\beta - \gamma)/n} \left[ \mu^{\boxplus n(n-1)/2} \right].
\end{equation}
\end{proposition}
\begin{corollary}
Let $W_d \in M_d(\mathbb C) \otimes M_n(\mathbb C)$ be a sequence of Wishart random matrices converging in distribution to a free Poisson distribution of parameter $\lambda$, $\pi_\lambda$. Then, the block modified random matrices $W_d^{\varphi} =  [\mathrm{id} \otimes \varphi](W_d)$ converge in distribution to a compound free Poisson distribution $ \pi_{\mu}$, where
$$\mu = \lambda \left[ \delta_{\alpha + (\beta + \gamma)/n} + \left( \frac{n(n+1)}{2}-1\right)\delta_{(\beta + \gamma)/n} + \frac{n(n-1)}{2} \delta_{(\beta - \gamma)/n} \right] .$$
\end{corollary}

The results above generalize the ones from Sections \ref{sec:partial-transposition} and \ref{sec:partial-reduction}, with the choice of parameters $\alpha = \beta = 0$, $\gamma =1$, resp.~ $\alpha = -1$, $\beta =1$, $\gamma = 0$.

\subsection{Mixtures of orthogonal conjugations}

Inspired by Weyl-covariant quantum channels (see, e.g.~ \cite{DFH06}), we consider maps $\varphi : M_n(\mathbb C) \to M_n(\mathbb C)$ of the form
\begin{equation}\label{eq:Weyl-covariant}
\varphi (A) = \sum_{i=1}^{n^2} \alpha_i U_i A U_i^*,
\end{equation}
where $\alpha_i \in \mathbb R$ and $U_i \in \mathcal U_n$ are \emph{orthogonal} unitary operators:
$$\mathrm{Tr}(U_iU_j^*) = n\delta_{ij}.$$
The Weyl-covariant channels from \cite[Section 4]{DFH06} are particular cases of the situation above, where $\{\alpha_i\}$ is a probability vector and the $U_i$ are the Weyl operators
$$W_{(x,y)} = U^xV^y$$
for $x,y \in \mathbb Z_d$ and $U,V$ being respectively the shift and phase operators
$$U E_x  = E_{x+1} \qquad \text{ and } \qquad V E_x = \exp\left(\frac{2\pi i x}{n} \right) E_x.$$

Maps $\varphi$ as in \eqref{eq:Weyl-covariant} are interesting because of the special form of their Choi matrix. Indeed, the Choi matrix reads
$$C_\varphi = \sum_{i=1}^{n^2} n\alpha_i u_iu_i^*,$$
where the $u_i$ are \emph{orthonormal} vectors in $\mathbb C^n \otimes \mathbb C^n$ given by
$$u_i = \frac{1}{\sqrt n} \sum_{s,t=1}^n U_i(s,t) E_s \otimes E_t.$$
Hence, the eigenvalues of $C_\varphi$ are $\{n \alpha_i\}$, they are simple, and the partial traces of the corresponding eigenprojectors are all equal to $n^{-1}I_n$. Thus, the Choi matrix $C_\varphi$ satisfies the unitarity condition (UC). We leave the proofs of the following results to the reader.

\begin{proposition}\label{prop:Weyl-covariant}
Consider a sequence of random matrices $X_d \in M_d(\mathbb C) \otimes M_n(\mathbb C)$ converging in distribution to ta probability measure $\mu$. Then, for the application $\varphi$ from \eqref{eq:Weyl-covariant}, the sequence of the block modified matrices $X_d^{\varphi}:= [\mathrm{id} \otimes \varphi](X_d) \in M_d(\mathbb C) \otimes M_n(\mathbb C)$ converges in distribution to the probability measure $\mu^{\varphi}$ given by
\begin{equation}
\mu^{\varphi} = \boxplus_{i=1}^{n^2}  D_{\alpha_i}[\mu].
\end{equation}
\end{proposition}
\begin{corollary}
Let $W_d \in M_d(\mathbb C) \otimes M_n(\mathbb C)$ be a sequence of Wishart random matrices converging in distribution to a free Poisson distribution of parameter $\lambda$, $\pi_\lambda$. Then, the block modified random matrices $W_d^{\varphi} =  [\mathrm{id} \otimes \varphi](W_d)$ converge in distribution to a compound free Poisson distribution $ \pi_{\mu}$, where
$$\mu = \lambda \sum_{i=1}^{n^2} \delta_{\alpha_i} .$$
\end{corollary}

In the case where the coefficients $\alpha_i$ are all equal to 1, we obtain the map $\varphi(A) = \mathrm{Tr}(A)I_n$, thus recovering the special case $\alpha = \gamma = 0$, $\beta = 1$ from Proposition \ref{prop:partial-generalization}.

\bigskip

\noindent \textit{Acknowledgments.} We would like to thank Teo Banica and Roland Speicher for useful discussions. I.N.'s research has been supported by a von Humboldt fellowship and by the ANR projects {OSQPI} {2011 BS01 008 01},  {RMTQIT}  {ANR-12-IS01-0001-01}, and {StoQ} {ANR-14-CE25-0003-01}. The authors acknowledge the support received from a Procope grant, which allowed them to meet in Saarbr\"ucken and in Toulouse on several occasions. O.A.~was supported by CONACYT GRANT No. 22668.

\bibliography{Blt}

\providecommand{\bysame}{\leavevmode\hbox to3em{\hrulefill}\thinspace}
\providecommand{\MR}{\relax\ifhmode\unskip\space\fi MR }
\providecommand{\MRhref}[2]{%
  \href{http://www.ams.org/mathscinet-getitem?mr=#1}{#2}
}
\providecommand{\href}[2]{#2}
\begin{thebibliography}{BSTV14}

\bibitem[AGZ10]{AGZ10}
Greg Anderson, Alice Guionnet, and Ofer Zeituni, \emph{An {I}ntroduction to
  {R}andom {M}atrices}, Cambridge University Press, Cambridge, 2010.

\bibitem[Aub12]{Au12}
Guillaume Aubrun, \emph{Partial transposition of random states and non-centered
  semicircular distributions.}, Random Matrices: Theory and Applications
  \textbf{01} (2012).

\bibitem[BB07]{BeBe07}
Serban Belinschi and Hari Bercovici, \emph{A new approach to subordination
  results in free probability}, Journal d'Analyse Mathematique \textbf{101}
  (2007), 357--365.

\bibitem[BG07]{benaych2007infinitely}
Florent Benaych-Georges, \emph{Infinitely divisible distributions for
  rectangular free convolution: classification and matricial interpretation},
  Probability Theory and Related Fields \textbf{139} (2007), no.~1-2, 143--189.

\bibitem[BG09a]{BG09}
Florent Benaych-Georges, \emph{Rectangular random matrices, related
  convolution}, Prob. Theory Rel. Field \textbf{144} (2009), 471--515.

\bibitem[BG09b]{BG09b}
\bysame, \emph{Rectangular random matrices, related free entropy and free
  {F}isher's information}, J. Operator Th. \textbf{62} (2009), no.~2, 371--419.

\bibitem[Bia98]{Bi98}
Philippe Biane, \emph{Processes with free increments}, Math. Z. \textbf{227}
  (1998), 143--174.

\bibitem[BMS13]{BMS13}
Serban Belinschi, Tobias Mai, and Roland Speicher, \emph{Analytic subordination
  theory of operator-valued free additive convolution and the solution of a
  general random matrix problem (pre-print)}, arXiv:1303.3196 (2013).

\bibitem[BN12a]{BaNe12}
Teodor Banica and Ion Nechita, \emph{Asymptotic eigenvalue distributions of
  block-transposed wishart matrices}, J. of Theoretical Probability \textbf{26}
  (2012), 855--869.

\bibitem[BN12b]{BaNe12b}
\bysame, \emph{Block-modified {W}ishart matrices and free {P}oisson laws},
  Houston J. Math. (2012).

\bibitem[BSTV14]{BSTV14}
Serban Belinschi, Roland Speicher, John Treilhard, and Carlos Vargas,
  \emph{Operator-valued free multiplicative convolution: analytic subordination
  theory and applications to random matrix theory}, Internat. Math. Res.
  Notices (2014).

\bibitem[CAG99]{CAG99}
NJ~Cerf, C~Adami, and RM~Gingrich, \emph{Reduction criterion for separability},
  Physical Review A \textbf{60} (1999), no.~2, 898.

\bibitem[CHN15]{CoHaNe15}
Benoit Collins, Patrick Hayden, and Ion Nechita, \emph{Random and free positive
  maps with applications to entanglement detection}, arXiv:1505.08042 (2015).

\bibitem[Cho75]{choi1975completely}
Man-Duen Choi, \emph{Completely positive linear maps on complex matrices},
  Linear Algebra and its Applications \textbf{10} (1975), no.~3, 285--290.

\bibitem[DFH06]{DFH06}
N~Datta, M~Fukuda, and AS~Holevo, \emph{Complementarity and additivity for
  covariant channels}, Quantum Information Processing \textbf{5} (2006), no.~3,
  179--207.

\bibitem[Dyk06]{Dy06}
Kenneth Dykema, \emph{On the ${S}$-transform over a {B}anach algebra}, J.
  Funct. Anal. \textbf{231} (2006), no.~1, 90--110.

\bibitem[FN15]{FuNe15}
Motohisa Fukuda and Ion Nechita, \emph{Additivity rates and ppt property for
  random quantum channels}, to appear in Ann. Math. Blaise Pascal (2015).

\bibitem[F{\'S}13]{fukuda2013partial}
Motohisa Fukuda and Piotr {\'S}niady, \emph{Partial transpose of random quantum
  states: Exact formulas and meanders}, Journal of Mathematical Physics
  \textbf{54} (2013), no.~4, 042202.

\bibitem[HHH96]{horodecki1996separability}
Micha{\l} Horodecki, Pawe{\l} Horodecki, and Ryszard Horodecki,
  \emph{Separability of mixed states: necessary and sufficient conditions},
  Physics Letters A \textbf{223} (1996), no.~1, 1--8.

\bibitem[HP00]{HiPe00}
Fumio Hiai and Dennes Petz, \emph{The semicircle law, free random variables,
  and entropy}, Mathematical surveys and monographs, vol.~77, American
  Mathematical Society, 2000.

\bibitem[JLN14]{JLN14}
Maria~Anastasia Jivulescu, Nicolae Lupa, and Ion Nechita, \emph{On the
  reduction criterion for random quantum states}, Journal of Mathematical
  Physics \textbf{55} (2014), no.~11, 112203.

\bibitem[JLN15]{JLN15}
\bysame, \emph{Thresholds for entanglement criteria in quantum information
  theory}, arXiv:1503.08008 (2015).

\bibitem[NS06]{NiSp06}
Alexandru Nica and Roland Speicher, \emph{Lectures on the combinatorics of free
  probability}, London Mathematical Society Lecture Note Series, vol. 335,
  Cambridge University Press, Cambridge, 2006.

\bibitem[NSS02]{NSS02}
Alexandru Nica, Dimitri Shlyakhtenko, and Roland Speicher,
  \emph{Operator-valued distributions {I}: {C}haracterizations of freeness},
  Internat. Math. Res. Notices \textbf{29} (2002), 1509--1538.

\bibitem[Shl96]{Sh96}
Dimitri Shlyakhtenko, \emph{Random {G}aussian band matrices and freeness with
  amalgamation}, Internat. Math. Res. Notices \textbf{20} (1996), 1013--1025.

\bibitem[Spe98]{Sp98}
Roland Speicher, \emph{Combinatorial theory of the free product with
  amalgamation and operator-valued free probability theory}, Memoirs of the
  American Math. Society, vol. 132, 1998.

\bibitem[Voi85]{Vo85}
Dan Voiculescu, \emph{Symmetries of some reduced free product
  {${C}^\ast$}-algebras}, Operator algebras and their connections with topology
  and ergodic theory ({B}u\c steni, 1983), Lecture Notes in Math., vol. 1132,
  Springer, Berlin, 1985, pp.~556--588.

\bibitem[Voi87]{Vo87}
\bysame, \emph{Multiplication of certain noncommuting random variables}, J.
  Operator Theory \textbf{18} (1987), 223--235.

\bibitem[Voi91]{Vo91}
\bysame, \emph{Limit laws for random matrices and free products}, Invent. Math.
  \textbf{104} (1991), 201--220.

\bibitem[Voi95]{Vo95}
\bysame, \emph{Operations on certain non-commutative operator-valued random
  variables. {R}ecent advances in operator algebras (orleans, 1992)},
  Asterisque \textbf{232} (1995).

\bibitem[Voi98]{Vo98}
\bysame, \emph{A strengthened asymptotic freeness result for random matrices
  with applications to free entropy}, Internat. Math. Res. Notices (1998),
  41--73.

\bibitem[Voi02]{Vo02}
\bysame, \emph{Analytic subordination consequences of free {M}arkovianity},
  Indiana Univ. Math. J. \textbf{51} (2002), 1161--1166.

\end{thebibliography}
\bibliographystyle{amsalpha}

\end{document}